\documentclass[
	framed_theorems,
	framed_proofs,
	abstract=true,
	print,
	english,
]{OCPreprint}

\usepackage{enumitem}
\setlist[enumerate]{label=(\alph*)}

\usepackage[boxed]{algorithm2e}
\SetAlCapSkip{.5em}

\crefname{algocf}{Algorithm}{Algorithms}

\renewcommand\phi\varphi
\newcommand\eps\varepsilon
\newcommand\Laplace\Delta

\newcommand\aeon{\;\text{a.e.\ on}\;}
\newcommand{\meas}[2][]{m\paren[#1]{#2}}
\newcommand{\Wad}{W_{\mathrm{ad}}}
\newcommand\dualh[3][]{\dual[#1]{#2}{#3}_{H^{-1}(\Omega)\times H_0^1(\Omega)}}
\newcommand\Cmin{C_{\min}}
\newcommand\lptimeslq{L^p(\Omega)\times L^q(\Omega)}

\definecolor{todocolor}{rgb}{1.0,0.3,0.3}

\newenvironment{minproblem}[2][]{%
	\begin{aligned}
		\min_{#1} \quad& #2
		\\
		\text{s.t.} \quad&
		\begin{aligned}[t]% the option t forces proper alignment
		}{%
		\end{aligned}
	\end{aligned}
}

\let\cite\parencite

\ifframedtheorems
	
\else
	
\fi

\allowdisplaybreaks

\author{Felix Harder\and Gerd Wachsmuth}
\title[M-stationarity for MPCCs in Lebesgue spaces]{%
M-stationarity for a class of MPCCs in Lebesgue spaces}

\date{October 21, 2021}

\begin{document}
%%fakesection: Title, abstract und co
\maketitle
\begin{abstract}
	We show that an optimality condition of M-stationarity type
	holds for minimizers of a class of
	mathematical programs with complementarity constraints (MPCCs) in Lebesgue spaces.
	We apply these results also to local
	minimizers of an inverse optimal control problem
	(which is an instance of an infinite-dimensional bilevel optimization problem).
	The multipliers for the M-stationarity system
	can be constructed
	via convex combinations of various multipliers to auxiliary, linear problems.
	However, proving the existence of the multipliers to these
	auxiliary problems is difficult and only possible in some situations.
\end{abstract}


\begin{msc}
	\mscLink{90C33}, % MPCCs in finite dimensions
	\mscLink{49K20}, % Optimality conditions for problems involving partial differential equations
	\mscLink{49K21}  % Optimality conditions for problems involving relations other than differential equations
\end{msc}

\section{Introduction}
\label{sec:intro}

In this article, we consider so-called 
mathematical programs with complementarity constraints
(MPCCs for short) in Lebesgue spaces,
see \eqref{eq:abstract_mpcc} for such an optimization problem.
For this type of problems, it is characteristic that they contain
a complementarity constraint, i.e.\ a constraint of the form
\begin{equation*}
	0 \le G(x)
	\quad\land\quad
	G(x) H(x) = 0
	\quad\land\quad
	H(x) \ge 0
	\qquad\aeon \Omega,
\end{equation*}
where $G:X\to L^p(\Omega)$, $H:X\to L^q(\Omega)$ are functions, 
$X$ is a Banach space and
$L^p(\Omega)$, $L^q(\Omega)$ are Lebesgue spaces over some domain $\Omega$.
Other (non-complementarity) constraints can involve Banach spaces
which are not necessarily Lebesgue spaces.

We are interested in first-order necessary stationarity conditions for such a class of 
optimization problems.

In finite-dimensional spaces, MPCCs have been studied frequently,
both from a numerical and theoretical perspective.
Even in finite-dimensional spaces, 
MPCCs form a challenging class of optimization problems.
Several problem-tailored stationarity conditions have been developed.
One such stationarity condition is the so-called strong stationarity
(which is equivalent to the KKT conditions).
However, strong stationarity does not need to hold for local minimizers even if
the data is linear.
A stationarity condition which holds under the very weak constraint qualification
MPCC-GCQ
(which, e.g., holds if the data is linear) is M-stationarity,
see \cite[Theorem~14]{FlegelKanzowOutrata2006}.
Therein, this result was proved by using the so-called
limiting normal cone and its calculus rules.

Recently, a simpler proof of M-stationarity for finite-dimensional MPCCs
has been found, see \cite{Harder2020}.
In particular, the concept of the limiting normal cone was not required.
Instead, the multipliers to the system of M-stationarity 
were constructed as a convex combination of several multipliers
which satisfy a system of A$_\beta$-stationarity
(where we use A$_\beta$-stationarity as defined in \cref{def:stat}).
These A$_\beta$-stationarity systems have already been observed
in \cite[Theorem~3.4]{FlegelKanzow2005:3}.

A first idea to prove M-stationarity for MPCCs in Lebesgue spaces
is to utilize the calculus for the limiting normal cone.
This fails for two reasons.
First, it was shown in 
\cite{MehlitzWachsmuth2016:1,MehlitzWachsmuth2017:1}
that the limiting normal cone to a complementarity set in $L^2(\mu)$
is always convex and corresponds to weak stationarity.
Second,
the calculus rules for the limiting normal cone typically utilize
the notion of \emph{sequential normal compactness},
but this fails for many subsets of Lebesgue spaces,
see \cite{Mehlitz2018}.

The situation is slightly different for complementarity sets
in Sobolev spaces.
In the one-dimensional situation,
it can be shown that the limiting normal cone is small enough
and, thus,
it is not too difficult to
prove
necessary optimality conditions of M-stationary type,
see
\cite{JarusekOutrata2008,Wachsmuth2014:2}.
In dimensions $n \ge 2$,
these arguments cannot be transferred,
see
\cite{HintermuellerMordukhovichSurowiec2011,Wachsmuth2014:2}.
The main reason was provided later in
the paper
\cite{HarderWachsmuth2017:1},
which shows that again the limiting normal cone
of a complementarity set in Sobolev spaces is too big
if the dimension is bigger than $1$.

In the present paper, we are going to transfer the approach of \cite{Harder2020}.
The idea is to take convex combinations of so-called A$_\beta$-stationary multipliers,
where $\beta$ ranges over the measurable subsets of the biactive set,
see \cref{def:stat}.
Under some regularity assumptions on these multipliers,
we were indeed able to prove the existence of an M-stationary multiplier.
A very crucial difference to the finite-dimensional situation
is that the existence of the A$_\beta$-stationary multipliers
does not come for free.
To explain this further,
the associated system of A$_\beta$-stationarity is precisely the KKT system
of a tightened and linearized problem 
(see \eqref{eq:lpbeta} for such a problem).
In finite dimensions, the existence of multipliers
follows directly from the linearity of the problem,
but this is not the case in infinite dimensions.
Quite surprisingly,
there exist situations in which solutions of this problem fail to possess
multipliers, see \cref{ex:no_fcq}.

Therefore,
the main contributions of this paper are the following.
\begin{enumerate}[label=(\roman*)]
	\item
		Combining A$_\beta$-stationary multipliers into M-stationary multipliers
		is possible
		(\cref{thm:schinabeck}).
	\item
		We prove the existence of A$_\beta$-stationary multipliers for a linear MPCC
		(\cref{prop:kkt_lp_beta}).
	\item
		By means of an example,
		we show that nonlinear MPCCs can be linearized
		(\cref{lem:linearized_OC})
		in order to apply
		the previous results
		(\cref{thm:mstat}).
\end{enumerate}
Thus, we give the first result which provides a system of M-stationarity
for an MPCC in Lebesgue spaces.
Previously,
it was unknown 
whether M-stationarity holds
for instances of MPCCs and bilevel problems in Lebesgue spaces,
see the open questions in
\cite[p.~642]{HarderWachsmuth2018:1}, \cite[p.~224]{Harder2021}.

In \cref{sec:preliminaries}, we formally introduce MPCCs in Lebesgue spaces
and define the corresponding stationarity conditions.
\Cref{sec:a_to_m} deals with the combination of multipliers
to the A$_\beta$-stationarity system into multipliers of the M-stationarity system.
In \cref{sec:linear_mpcc} we apply this to a class of linear MPCCs.
It turns out that showing A$_\beta$-stationarity of local minimizers
is difficult and is only possible under some further assumptions on the structure
of the MPCC.
In \cref{sec:IOC} we study
an inverse optimal control problem, which is nonlinear and more natural looking
than a linear MPCC.
The KKT reformulation of the inverse optimal control problem
is an MPCC problem in Lebesgue spaces and thus we can apply the previous results.
However, the linearization of the inverse optimal control problem
requires some effort, including regularization methods.
Finally, in \cref{sec:conclusion} we give an outlook and conclude.

\section{Stationarity concepts for MPCCs in Lebesgue spaces}
\label{sec:preliminaries}

We first consider an abstract class
of MPCCs in Lebesgue spaces.
This class is broad enough such that
all MPCCs from this article fit into the setting.

Let $(\Omega_i,\AA_i,m_i)$ be a measure space
for all $i\in\set{1,2}$ and let $X,Y$ be Banach spaces.
We consider the class of MPCCs
given by
\begin{equation*}
	\label{eq:abstract_mpcc}
	\tag{MPCC}
	\begin{minproblem}[x \in X]{F(x)}
		h(x) &= 0,
		\\
		g(x) &\geq0 \quad\aeon\Omega_1,
		\\
		0 &\leq G(x) \perp H(x)\geq0
		\quad\aeon\Omega_2,
	\end{minproblem}
\end{equation*}
where $F \colon X\to\R$, $h \colon X\to Y$,
$g \colon X\to L^2(\Omega_1)$, $G,H \colon X\to L^2(\Omega_2)$
are continuously Fréchet differentiable functions.
The last constraint in \eqref{eq:abstract_mpcc}
is an abbreviation of
the pointwise complementarity
\begin{equation*}
	0 \le G(x)
	\quad\land\quad
	G(x) H(x) = 0
	\quad\land\quad
	H(x) \ge 0
	\qquad\aeon \Omega_2.
\end{equation*}

For a measurable function $v \colon \Omega_i\to\R$, we will use the notation
$\set{v>0}$ to refer to the set
$\set{\omega\in\Omega_i\given v(\omega)>0}$,
and similar notations for other pointwise relations
will be used throughout this paper.

Let us define stationarity conditions for this MPCC.
\begin{definition}
	\label{def:stat}
	Let $\bar x$ be a feasible point of 
	\eqref{eq:abstract_mpcc}.
	We call $\bar x$
	a \emph{weakly stationary point}
	if there exist multipliers
	$\bar\eta\in Y\dualspace$, $\bar\rho\in L^2(\Omega_1)$,
	$\bar\mu,\bar\nu\in L^2(\Omega_2)$ such that
	\begin{align*}
		F'(\bar x)
		+ h'(\bar x)\adjoint\bar\eta
		+ g'(\bar x)\adjoint\bar\rho
		+ G'(\bar x)\adjoint\bar\mu
		+ H'(\bar x)\adjoint\bar\nu
		&= 0,
		\\
		\bar\rho&\leq0\quad\aeon\Omega_1,
		\\
		\bar\rho &=0 \quad\aeon\set{g(\bar x)>0},
		\\
		\bar\mu &= 0 \quad\aeon\set{G(\bar x)>0},
		\\
		\bar\nu &= 0 \quad\aeon\set{H(\bar x)>0}.
	\end{align*}
	If we have the additional condition
	\begin{equation*}
		(\bar\mu<0\land\bar\nu<0)\lor\bar\mu\bar\nu=0
		\qquad \aeon \set{G(\bar x)=H(\bar x)=0}
	\end{equation*}
	then $\bar x$ is an \emph{M-stationary point}.
	We say that $\bar x$ is a strongly stationary
	(or S-stationary)
	point if $\bar x$ is weakly stationary and
	\begin{equation*}
		\bar\mu\leq0\land\bar\nu\leq0
		\qquad \aeon\set{G(\bar x)=H(\bar x)=0}
	\end{equation*}
	holds.
	We say that $\bar x$ is \emph{A$_\beta$-stationary}, where 
	$\beta \subset \set{G(\bar x)=H(\bar x)=0}$
	is measurable,
	if
	there exist weakly stationary multipliers which additionally satisfy
	\begin{equation*}
		\bar\mu \le 0 \quad\aeon \set{G(\bar x)=H(\bar x)=0} \setminus \beta
		\qquad\text{and}\qquad
		\bar\nu \le 0 \quad\aeon \beta
		.
	\end{equation*}
	If there exists a measurable $\beta \subset \set{G(\bar x)=H(\bar x)=0}$
	such that $\bar x$ is A$_\beta$-stationary,
	then $\bar x$ is called \emph{A-stationary}.
	If $\bar x$ is A$_\beta$-stationary
	for all measurable $\beta \subset \set{G(\bar x)=H(\bar x)=0}$,
	we say that $\bar x$ is \emph{A$_\forall$-stationary}.
\end{definition}

In the case that $\Omega_i$ is finite (and equipped with the counting measure $m_i$)
and $X,Y$ are finite dimensional,
\cref{def:stat} coincides with the
definitions for finite-dimensional MPCCs in the literature.
An exception is A$_\forall$-stationarity, which was not yet defined
for MPCCs, but is equivalent to linearized B-stationarity,
see \cite[Proposition~3.1~(b)]{Harder2021:1}.
We further note that (as far as we know) A$_\beta$-stationarity was only
recently defined in finite dimensions in \cite[Definition~2.5]{Harder2021:1}.

We also mention that in an infinite-dimensional setting,
the name ``M-stationarity'' can sometimes refer to a condition
involving the limiting normal cone (see, e.g., \cite[Definition~3.2]{Mehlitz2017}),
which would be different from our pointwise a.e.\ definition.

\section{From A$_\forall$-stationarity to M-stationarity}
\label{sec:a_to_m}

In \cite[Lemma~3.2]{Harder2020}, it was shown for finite dimensional MPCCs that
one can construct multipliers for the system of M-stationarity
as a convex combination of the collection of 
multipliers which satisfy a system of A$_\beta$-stationarity for some $\beta$.
In this \lcnamecref{sec:a_to_m}, we want to make a similar statement
in an infinite-dimensional setting.
As we do not have finitely many multipliers anymore,
this requires some additional assumptions, and also the proof is much more complicated.
We formulate the main \lcnamecref{thm:schinabeck} of this
\lcnamecref{sec:a_to_m}.
Its proof is given after some auxiliary lemmas.
\begin{theorem}
	\label{thm:schinabeck}
	Let $(\Omega,\AA,m)$ be a measure space
	and $1<p,q<\infty$.
	Suppose that for all $\alpha\in\AA$
	there exist
	$\mu^\alpha\in L^p(\Omega),\nu^\alpha\in L^q(\Omega)$
	with
	\begin{align}
		\label{eq:mu_nu_leq1}
		\mu^\alpha &\leq0\quad\aeon\Omega\setminus\alpha,
		&
		\nu^\alpha &\leq0\quad\aeon\alpha.
	\end{align}
	Furthermore, we assume the existence of a measurable function 
	$c_0:\Omega\to\R$
	and a constant $c_1>0$ with
	\begin{subequations}
		\label{eq:bounds_on_mu_nu}
		\begin{align}
			\max(\mu^\alpha,\nu^\alpha) &\leq c_0
			\qquad\aeon\Omega,
			\quad\forall \alpha\in\AA,
			\label{eq:bounds_on_mu_nu_1}
			\\
			\norm{\mu^\alpha}_{L^p(\Omega)}+\norm{\nu^\alpha}_{L^q(\Omega)}
			&\leq c_1
			\qquad\forall \alpha\in\AA.
			\label{eq:bounds_on_mu_nu_2}
		\end{align}
	\end{subequations}
	Then, there exists a point 
	\begin{equation*}
		(\bar\mu,\bar\nu) \in
		B:= \clconv\set{ (\mu^\alpha,\nu^\alpha)
		\given \alpha\in\AA }\subset L^p(\Omega)\times L^q(\Omega)
	\end{equation*}
	which satisfies
	\begin{equation}
		\label{eq:mstat}
		(\bar\mu<0\land\bar\nu<0)\lor\bar\mu\bar\nu=0
		\qquad\aeon\Omega.
	\end{equation}
\end{theorem}
Here and in the sequel,
$\clconv D$ denotes the closed convex hull of a set $D$.

If $\Omega$ is finite (and if $\AA$ is the power set of $\Omega$), 
then \eqref{eq:bounds_on_mu_nu} is automatically satisfied, 
and the statement is the same as \cite[Lemma~3.2]{Harder2020}.

Later, we are going to apply this \lcnamecref{thm:schinabeck}
to an instance of \eqref{eq:abstract_mpcc}
with the setting
$\Omega = \set{G(\bar x) = H(\bar x) = 0}$.
Then, required existence of multipliers
satisfying \eqref{eq:mu_nu_leq1}
is basically the A$_\forall$-stationarity of $\bar x$
(with the additional regularity \eqref{eq:bounds_on_mu_nu}),
whereas the constructed multipliers
$(\bar \mu, \bar \nu)$
will yield the M-stationarity of $\bar x$,
see \cref{thm:mstat}.

The proof requires several steps.
First, for all $\alpha\in\AA$ we define the sets
\begin{equation*}
	A^\alpha :=
	\set{(\mu,\nu)\in L^p(\Omega)\times L^q(\Omega)
		\given
		\mu \leq 0\aeon\Omega\setminus\alpha,\,
		\nu\leq0\aeon\alpha
	}
\end{equation*}
which describe \eqref{eq:mu_nu_leq1}.
Clearly, we have $(\mu^\alpha,\nu^\alpha)\in B\cap A^\alpha$
for all $\alpha\in\AA$.
We also define the set of sets
\begin{equation*}
	\CC:=\set[\bigg]{
		\clconv\set{x^\alpha\given \alpha\in\AA }
		\given
		\set{x^\alpha}_{\alpha \in \AA} \subset B
		:
		\forall \alpha\in\AA :
		x^\alpha\in A^\alpha
	}\subset \PP(L^p(\Omega)\times L^q(\Omega)).
\end{equation*}
Clearly we have $B\in\CC$.
Note that all sets in $\CC$ are closed, convex and nonempty subsets of $B$.
As $B$ is bounded due to \eqref{eq:bounds_on_mu_nu_2}, all sets in $\CC$ are bounded.
We are interested in a set in $\CC$
which is minimal with respect to set inclusion.
The next two \lcnamecrefs{lem:compact_chain} prove the existence of such a minimal set.
\begin{lemma}
	\label{lem:compact_chain}
	Let $I$ be a nonempty set and let $Y$ be a topological space.
	Suppose there exists a nonempty compact set $D_i\subset Y$ for each $i\in I$
	and that these sets satisfy
	\begin{equation}
		\label{eq:subsets_total_order}
		D_i\subset D_j\lor D_j\subset D_i
		\qquad\forall i,j\in I.
	\end{equation}
	Then
	$ \bigcap_{i\in I} D_i\neq\emptyset $
	holds.
\end{lemma}
\begin{proof}
	Pick $i_0\in I$ arbitrary.
	Suppose that $\bigcap_{i\in I} D_i=\emptyset$.
	Then the family $\set{Y\setminus D_i}_{i\in I}$ of open sets 
	covers $Y$, and in particular $D_{i_0}$.
	Thus, we can select a finite subcover
	$\set{Y\setminus D_{i_j}}_{j=1}^n$ which covers $D_{i_0}$.
	Due to \eqref{eq:subsets_total_order}
	we can without loss of generality assume
	that $D_{i_1}\subset D_{i_2}\subset\ldots\subset D_{i_n}$.
	Thus, the single open set $Y\setminus D_{i_1}$ is a cover of $D_{i_0}$.
	This implies $D_{i_1}\cap D_{i_0}=\emptyset$,
	which contradicts \eqref{eq:subsets_total_order}.
\end{proof}
\begin{lemma}
	\label{lem:minimal_set}
	There exists $\Cmin\in\CC$
	which is minimal with respect to set inclusion.
\end{lemma}
\begin{proof}
	We will use Zorn's Lemma on $\CC$
	with the partial ordering ``$\subset$''.
	Clearly, $\CC$ is nonempty.
	Let $\set{C_i}_{i\in I}\subset \CC$ be a nonempty totally ordered chain.
	We want to show that this chain has a lower bound.
	It follows that $\set{A^\alpha\cap C_i}_{i\in I}$
	is a totally ordered chain for each $\alpha\in\AA$.
	Because $A^\alpha\cap C_i\subset B$ is a bounded and closed set, 
	it is weakly compact for all $\alpha\in\AA$, $i\in I$.
	Moreover, these sets are nonempty.
	Thus, for each $\alpha$, we can apply \cref{lem:compact_chain}
	for the weak topology on $L^p(\Omega)\times L^q(\Omega)$
	to the family of sets $\set{A^\alpha\cap C_i}_{i\in I}$.
	This yields  that the intersection is nonempty,
	i.e.\ there exists a point $\hat x^\alpha$
	with $\hat x^\alpha\in A^\alpha\cap C_i$ for all $i\in I$.
	Now we set $C_0:=\clconv\set{\hat x^\alpha\given \alpha\in\AA}$.
	Clearly, $C_0\in\CC$.
	Since the points $\hat x^\alpha$ are contained in $C_i$ for all $i\in I$,
	it follows that $C_0\subset C_i$ for all $i\in I$.
	Thus, the totally ordered chain $\set{C_i}_{i\in I}$
	has a lower bound in $\CC$.

	By Zorn's Lemma, there exists a set $\Cmin\in\CC$
	which is minimal with respect to the relation ``$\subset$''.
\end{proof}
It is clear that $\Cmin$ does not have to be unique.
By construction of $\CC$, the set $\Cmin$ is non-empty, closed and convex
and it is bounded due to
\eqref{eq:bounds_on_mu_nu_2}.
Owing to the Krein--Milman theorem,
we know that $\Cmin$ possesses extreme points.
For our our proof of \cref{thm:schinabeck},
we need the stronger notion of
so-called \emph{strongly exposed points} of $\Cmin$.
We start with the definition.
\begin{definition}
	\label{def:strongly_exposed_point}
	Let $D\subset X$, and $X$ a normed space.
	Then $x\in D$ is a strongly exposed point if there exists
	$x\dualspace\in X\dualspace$ such that
	$x$ maximizes $x\dualspace$ over $D$ and
	\begin{equation*}
		\dual{x\dualspace}{x_k}\to \dual{x\dualspace}{x}
		\implies
		x_k\to x
	\end{equation*}
	holds for all sequences $\seq{x_k}_{k\in\N}\subset D$.
\end{definition}
It can be seen that a strongly exposed point is always an extreme point.
A non-empty, convex, closed and bounded set in reflexive spaces always has a strongly
exposed point, see \cite[Theorem~8.28]{FabianHabalaHajekAn2001}.
In particular, $\Cmin$ possesses strongly exposed points.
Our next goal is to show that such a strongly exposed point of $\Cmin$
satisfies \eqref{eq:mstat},
which would prove \cref{thm:schinabeck}.
We give a preparatory \lcnamecref{lem:minimize_xi}.
\begin{lemma}
	\label{lem:minimize_xi}
	Let $\xi\in \paren[\big]{\lptimeslq}\dualspace$ 
	be a bounded linear functional on $\lptimeslq$.
	Then there exists a collection 
	$\set{\hat x^\alpha}_{\alpha\in\AA}$
	such that
	\begin{align}
		\hat x^\alpha &\in \Cmin\cap A^\alpha
		&&\forall \alpha\in\AA,
		\\
		\label{eq:xi_min}
		\dual\xi{\hat x^\alpha} &\leq \dual\xi{y^\alpha}
		&&\forall y^\alpha\in \Cmin\cap A^\alpha
		\quad\forall \alpha\in \AA,
		\\
		\label{eq:convclosure}
		\clconv\set{\hat x^\beta\given \beta\in\AA}
		&= \Cmin.
	\end{align}
\end{lemma}
\begin{proof}
	We choose $\hat x^\alpha$ as a minimizer of $\xi$
	over $\Cmin\cap A^\alpha$.
	Such a minimizer exists because
	$\Cmin\cap A^\alpha$ is weakly compact and $\xi$
	is weakly continuous.

	The condition \eqref{eq:convclosure} follows because
	all points $\hat x^\beta$ are in the closed and convex set
	$\Cmin$ and because $\Cmin$ was chosen as a minimal
	set in $\CC$ with respect to set inclusion.
\end{proof}

\begin{lemma}
	\label{lem:strongly_exposed_is_M}
	Every strongly exposed point $(\bar\mu,\bar\nu)$ of $\Cmin$ satisfies 
	\eqref{eq:mstat}.
\end{lemma}
\begin{proof}
	Let $\bar x=(\bar\mu,\bar\nu)$ be a strongly exposed point of $\Cmin$
	and let $\xi\in \paren[\big]{\lptimeslq}\dualspace$ 
	be the corresponding multiplier.
	We choose $\set{\hat x^{\alpha}}_{\alpha\in\AA}$
	according to \cref{lem:minimize_xi}.

	Let $k\in\N$ be given and suppose that
	$\dual\xi{\bar x}\geq \dual\xi{\hat x^\alpha}+\tfrac1k$ holds
	for all $\alpha\in\AA$.
	Then it would follow from \eqref{eq:convclosure}
	that $\xi$ separates the closed convex set $\Cmin$ from
	$\bar x$ and therefore $\bar x\not\in \Cmin$,
	which is not possible since $\bar x$ is an extremal point
	of $\Cmin$.
	Thus, there exists a set $\alpha_k\in\AA$ such that
	\begin{equation*}
		\dual\xi{\bar x} -\tfrac1k
		\leq \dual\xi{\hat x^{\alpha_k}}
		\leq \dual\xi{\bar x}
	\end{equation*}
	holds.
	Let $\ell\in\N$ with $\ell\geq2$ be fixed.
	We make the definitions
	\begin{align*}
		D_\ell &:= 
		\set{(a_1,a_2)\in\R^2\given
			\min(\abs{x_1},\abs{x_2})>1/\ell \land x_1x_2<0
		},
		\\
		\beta_k &:= \set{\omega\in\Omega\given
			\hat x^{\alpha_k}(\omega)\in D_\ell \land c_0(\omega)\leq \ell
		}
		\in\AA,
		\\
		\gamma_k &:= \alpha_k\Delta\beta_k\in\AA
	\end{align*}
	for all $k\in\N$,
	where $\Delta$ denotes the symmetric difference.
	We claim that
	\begin{equation*}
		y_k := (1-\ell^{-3})\hat x^{\alpha_k}+\ell^{-3} \hat x^{\gamma_k}
		\in A^{\alpha_k}\cap \Cmin
	\end{equation*}
	holds for all $k\in\N$.
	Indeed, for $\omega\in\alpha_k\cap \beta_k$ we have
	$ y_k(\omega)_2
	\leq (1-\ell^{-3})\hat x^{\alpha_k}(\omega)_2+ \ell^{-3}c_0(\omega)
	\leq (1-\ell^{-3})(-\ell^{-1}) + \ell^{-3} \ell \leq 0 $
	a.e.\ on $\alpha_k\cap\beta_k$.
	Similarly, for $\omega\in\beta_k\setminus\alpha_k$ we have
	$ y_k(\omega)_1
	\leq (1-\ell^{-3})\hat x^{\alpha_k}(\omega)_1+ \ell^{-3}c_0(\omega)
	\leq (1-\ell^{-3})(-\ell^{-1}) + \ell^{-3} \ell \leq 0 $
	a.e.\ on $\beta_k\setminus\alpha_k$.
	If $\omega\not\in\beta_k$, then we have
	$\omega\in\alpha_k\Leftrightarrow\omega\in\gamma_k$,
	which implies $y_k(\omega)_1\leq0$ a.e.\ on 
	$\Omega\setminus(\alpha_k\cup\beta_k)$
	and $y_k(\omega)_2\leq0$ a.e.\ on $\alpha_k\setminus\beta_k$
	by convexity and $\hat x^{\gamma_k}\in A^{\gamma_k}$.
	This shows $y_k\in A^{\alpha_k}$,
	and $y_k\in \Cmin$ follows by convexity.

	According to \eqref{eq:xi_min} we then have
	$\dual\xi{\hat x^{\alpha_k}}\leq\dual\xi{ y^k}$, which implies
	\begin{equation*}
		\dual\xi{\bar x} -\tfrac1k
		\leq \dual\xi{\hat x^{\alpha_k}}
		\leq \dual\xi{\hat x^{\gamma_k}}
		\leq \dual\xi{\bar x}.
	\end{equation*}
	Then the fact that $\bar x$ is a strongly exposed point of
	$\Cmin$ implies the convergences $\hat x^{\alpha_k}\to \bar x$ and
	$\hat x^{\gamma_k}\to\bar x$.
	Next,
	using $\hat x^{\alpha_k}\in A^{\alpha_k}$,
	$\hat x^{\gamma_k}\in A^{\gamma_k}$,
	and $\hat x^{\alpha_k}\in D_\ell$ a.e.\ on $\beta_k$,
	we observe that
	\begin{equation*}
		\begin{aligned}
			\hat x^{\alpha_k}_2 \leq0
			\;\land\;
			\hat x^{\alpha_k}_1 > 1/\ell
			\;\land\;
			\hat x^{\gamma_k}_1 \leq0
			\qquad\aeon\beta_k\cap\alpha_k,
			\\
			\hat x^{\alpha_k}_1 \leq0
			\;\land\;
			\hat x^{\alpha_k}_2 > 1/\ell
			\;\land\;
			\hat x^{\gamma_k}_2 \leq0
			\qquad\aeon\beta_k\setminus\alpha_k
		\end{aligned}
	\end{equation*}
	holds.
	This implies
	\begin{equation*}
		\ell\max\paren[\big]{\abs{\hat x^{\alpha_k}(\omega)_1-\hat x^{\gamma_k}(\omega)_1}
		,
	\abs{\hat x^{\alpha_k}(\omega)_2-\hat x^{\gamma_k}(\omega)_2}}
		\geq 1
		\quad\aeon\beta_k.
	\end{equation*}
	By integrating, we obtain
	\begin{equation*}
		\begin{aligned}
			\meas{\beta_k}
			&=
			\int_{\beta_k} 1 \,\d m
			\leq
			\int_{\beta_k} \ell^p \abs{\hat x^{\alpha_k}_1-\hat x^{\gamma_k}_1}^p \,\d m
			+ 
			\int_{\beta_k} \ell^q \abs{\hat x^{\alpha_k}_2-\hat x^{\gamma_k}_2}^q \,\d m
			\\
			&\leq
			\ell^p\norm{\hat x^{\alpha_k}_1-\hat x^{\gamma_k}_1}_{L^p(\Omega)}^p
			+
			\ell^q\norm{\hat x^{\alpha_k}_2-\hat x^{\gamma_k}_2}_{L^q(\Omega)}^q
			\to0
			\qquad\text{as}\;k\to\infty,
		\end{aligned}
	\end{equation*}
	where the last convergence follows from
	$\hat x^{\alpha_k}\to\bar x$ and $\hat x^{\gamma_k}\to\bar x$.
	Without loss of generality we can say
	that the convergence $\hat x^{\alpha_k}\to\bar x$ also holds
	pointwise a.e.\ (otherwise one can choose a subsequence with this property).
	Since $D_\ell$ is open, 
	$\chi_{\set{\bar x\in D_\ell}}
	\leq\liminf_{k\to\infty}\chi_{\set{\hat x^{\alpha_k}\in D_\ell}}$
	and therefore
	$\chi_{\set{\bar x\in D_\ell}\cap \set{c_0\leq\ell}}
	\leq\liminf_{k\to\infty}\chi_{\beta_k}$
	hold a.e.\ in $\Omega$.
	Applying Fatou's Lemma to the last inequality yields
	\begin{equation*}
		\meas[\big]{\set{\bar x\in D_\ell}\cap \set{c_0\leq\ell}}
		\leq \liminf_{k\to\infty} \meas{\beta_k}
		=0.
	\end{equation*}
	Since $\ell\in\N$, $\ell\geq2$ was arbitrary,
	this implies
	\begin{equation*}
		\begin{aligned}
			0 &=
			\meas[\bigg]{
				\bigcup_{\ell=2}^\infty
				\set{\bar x\in D_\ell}\cap \set{c_0\leq\ell}
			}
			= \meas[\bigg]{
				\paren[\Big]{
					\bigcup_{\ell=2}^\infty \set{\bar x\in D_\ell}
				}
				\cap
				\paren[\Big]{
					\bigcup_{\ell=2}^\infty \set{c_0\leq\ell}
				}
			}
			\\ &= 
			\meas[\bigg]{\paren[\Big]{ 
			\bigcup_{\ell=2}^\infty \set{\bar x\in D_\ell}\cap\Omega }}
			= \meas[\bigg]{ \bigcup_{\ell=2}^\infty \set{\bar x\in D_\ell} }
			= \meas{\set{\bar x_1 \bar x_2 < 0}},
		\end{aligned}
	\end{equation*}
	where we used that $\set[\big]{\set{\bar x\in D_\ell}}_{\ell\in\N}$
	and $\set[\big]{\set{c_0\leq\ell}}_{\ell\in\N}$
	are monotonically increasing sequences.
	Thus, we have $\bar x_1 \bar x_2 \ge 0$ a.e.\ on $\Omega$.
	Due to $\hat x^{\alpha_k}\in A^{\alpha_k}$ we have
	$\min(\hat x^{\alpha_k}_1,\hat x^{\alpha_k}_2)\leq0$ a.e.\ on $\Omega$,
	and by considering a pointwise a.e.\ converging subsequence
	of $\set{\hat x^{\alpha_k}}_{k\in\N}$, we obtain
	$\min(\bar x_1,\bar x_2)\leq0$ a.e.\ on $\Omega$.
	In summary, we have
	$\min(\bar\mu,\bar\nu)\leq0$
	and $\bar\mu\bar\nu\geq0$ a.e.\ on $\Omega$.
	However, this is equivalent to \eqref{eq:mstat}.
\end{proof}
Now we have all the necessary ingredients to prove \cref{thm:schinabeck}.
\begin{proof}[Proof of \cref{thm:schinabeck}]
	As the space $\lptimeslq$ is reflexive,
	the set $\Cmin$ is a weakly compact subset of $\lptimeslq$. 
	Thus, according to \cite[Theorem~8.28]{FabianHabalaHajekAn2001}
	the set $\Cmin$ is the closed convex hull of its strongly exposed points.
	Since $\Cmin$ is nonempty, this implies that $\Cmin$ has at least one
	strongly exposed point.
	By \cref{lem:strongly_exposed_is_M}, this points satisfies \eqref{eq:mstat}
	and our proof of \cref{thm:schinabeck} is complete.
\end{proof}

\begin{remark}
	\label{rem:more_multipliers}
	We mention that \cref{thm:schinabeck}
	not only yields the existence of a single multiplier pair
	satisfying \eqref{eq:mstat}.
	Indeed, every strongly exposed point of $\Cmin$
	satisfies \eqref{eq:mstat}
	and $\Cmin$ is the closed convex hull of its strongly exposed points.
	Moreover, $\Cmin$ has a non-empty intersection with all $A^\alpha$.

	To summarize, the closed convex hull of all pairs $(\hat\mu, \hat\nu) \in B$
	satisfying \eqref{eq:mstat}
	has a non-empty intersection with $A^\alpha$ for all $\alpha \in \AA$.

	Furthermore, suppose that there exists no point $(\bar\mu,\bar\nu)\in B$ 
	with $\bar\mu\leq0$ and $\bar\nu\leq0$ a.e.\ in $\Omega$
	(which corresponds to strong stationarity).
	This means $\bigcap_{\alpha\in\AA} A^{\alpha}\cap B=\emptyset$,
	and therefore there must exist at least two points 
	$(\bar\mu,\bar\nu)\in B$ that satisfy \eqref{eq:mstat}.
\end{remark}

The assumptions in \cref{thm:schinabeck} require that there is a bound
on $\norm{\mu^\alpha}_{L^p(\Omega)},\norm{\nu^\alpha}_{L^q(\Omega)}$.
The next result is based on \cref{thm:schinabeck} but has different assumptions.
In particular, it could also be applied in situations where 
$\norm{\mu^\alpha}_{L^p(\Omega)},\norm{\nu^\alpha}_{L^q(\Omega)}$
might be unbounded, provided that some other assumptions are satisfied.
\begin{lemma}
	\label{lem:preprocessing}
	Let $(\Omega,\AA,m)$ be a measure space
	and $Y$ be a reflexive Banach space.
	Suppose that for all $\alpha\in\AA$
	there exist
	$\mu^\alpha,\nu^\alpha\in L^2(\Omega)$, $\rho^\alpha\in Y$
	such that
	\begin{align}
		\label{eq:mu_nu_leq}
		\mu^\alpha &\leq0\quad\aeon\Omega\setminus\alpha,
		&
		\nu^\alpha &\leq0\quad\aeon\alpha.
	\end{align}
	Furthermore, we assume that there exist functions $c_0\in L^2(\Omega)$,
	$c_1 \colon \AA\to(0,\infty)$
	with
	\begin{align}
		\label{eq:partial_boundedness}
		\max(\abs{\mu^\alpha+\nu^\alpha},\nu^\alpha) &\leq \mrep{c_0 }{c_1(\alpha) c_0}
		\quad\aeon\Omega
		\quad\forall\alpha\in\AA,
		\\
		\label{eq:alpha_dependent_bound}
		\abs{\nu^\alpha} &\leq c_1(\alpha) c_0
		\quad\aeon\Omega
		\quad\forall\alpha\in\AA,
	\end{align}
	and a bounded linear operator $T \colon L^2(\Omega)\to Y$
	and constant $c_2>0$ with
	\begin{equation}
		\label{eq:rho_bounded}
		\norm{\rho^\alpha-T\nu^\alpha}_Y \leq c_2
		\quad\forall\alpha\in\AA.
	\end{equation}
	Then there exists a point
	\begin{equation*}
		(\bar\mu,\bar\nu,\bar\rho) \in
		B:= \clconv\set{ (\mu^\alpha,\nu^\alpha,\rho^\alpha)
		\given \alpha\in\AA}
	\end{equation*}
	such that $\bar\mu$, $\bar\nu$ satisfy \eqref{eq:mstat}.
\end{lemma}
\begin{proof}
	We want to apply \cref{thm:schinabeck}.
	For that purpose, we want to modify $\mu^\alpha,\nu^\alpha$
	in such a way that they are bounded in $L^2(\Omega)$.
	We define the sets
	\begin{align*}
		A^\alpha &:= \set{(\mu,\nu,\rho)\in L^2(\Omega)^2\times Y\given
			\mu\leq0\aeon\Omega\setminus\alpha,\;
		\nu\leq0\aeon\alpha},
		\\
		B &:= \clconv\set{(\mu^\alpha,\nu^\alpha,\rho^\alpha)
		\given \alpha\in\AA}.
	\end{align*}
	For convexity and continuity reasons we also have the properties
	\begin{align}
		\label{eq:B_mu_nu_estimate}
		\max(\abs{\mu+\nu},\nu)&\leq c_0
		\quad\aeon\Omega
		\quad\forall (\mu,\nu,\rho)\in B,
		\\
		\label{eq:B_rho_estimate}
		\norm{\rho-T\nu}&\leq c_2
		\qquad\forall (\mu,\nu,\rho)\in B.
	\end{align}

	Let $\alpha\in\AA$ be given.
	We define 
	\begin{equation}
		\label{eq:c_3_def}
		c_3(\alpha):=\inf\set{d>0\given
			\exists (\mu,\nu,\rho)\in B\cap A^\alpha:
		\abs{\nu}\leq dc_0\aeon\Omega}.
	\end{equation}
	We claim that there exists 
	$(\hat\mu^\alpha,\hat\nu^\alpha,\hat\rho^\alpha)\in B\cap A^\alpha$
	with
	\begin{equation}
		\label{eq:hat_mu_alpha_intro}
		\abs{\hat\nu^\alpha}\leq c_3(\alpha)c_0
		\quad\aeon\Omega.
	\end{equation}
	Indeed, since $c_3(\alpha)\leq c_1(\alpha)<\infty$ we find
	sequences $\set{d_k}_{k\in\N}\subset(0,\infty)$,
	$\set{(\mu^k,\nu^k,\rho^k)}_{k\in\N}\subset B\cap A^\alpha$
	with $d_k\downto c_3(\alpha)$ and
	$\abs{\nu^k}\leq d_kc_0$ a.e.\ on $\Omega$.
	Then $\set{\nu^k}_{k\in\N}$ is a bounded sequence.
	By \eqref{eq:B_mu_nu_estimate} and \eqref{eq:B_rho_estimate}
	it follows that $\set{(\mu^k,\nu^k,\rho^k)}_{k\in\N}$
	is also a bounded sequence 
	and therefore has a weakly convergent subsequence.
	Let $(\hat\mu^\alpha,\hat\nu^\alpha,\hat\rho^\alpha)$
	denote a weak limit of a subsequence.
	Then $(\hat\mu^\alpha,\hat\nu^\alpha,\hat\rho^\alpha)\in B\cap A^\alpha$
	holds since $B\cap A^\alpha$ is a closed convex set.
	Similarly, for all $k\in\N$ the condition
	$\abs{\hat\nu^\alpha}\leq d_kc_0$ holds a.e.\ on $\Omega$.
	Then \eqref{eq:hat_mu_alpha_intro} follows from the convergence
	$d_k\downto c_3(\alpha)$.

	Our next goal will be to show that $c_3(\alpha)\leq 1$ holds.
	Suppose, by contradiction, that $c_3(\alpha)>1$ holds.
	Combining \eqref{eq:B_mu_nu_estimate} with
	$(\hat\mu^\alpha,\hat\nu^\alpha,\hat\rho^\alpha)\in B\cap A^\alpha$
	yields the conditions
	\begin{align*}
		\hat\nu^\alpha &\geq -\hat\mu^\alpha-c_0
		\geq -c_0
		\quad\aeon\Omega\setminus\alpha,
		\\
		\hat\nu^\alpha &\leq c_0
		\quad\aeon\Omega.
	\end{align*}
	For $t\in[0,1]$ we make the definitions
	\begin{align*}
		\beta &:= \alpha \cap \set{\hat\nu^\alpha\geq-c_0}\in\AA,
		\\
		(\mu^t,\nu^t,\rho^t) &:=
		(1-t)(\hat\mu^\alpha,\hat\nu^\alpha,\hat\rho^\alpha)
		+ t(\mu^\beta,\nu^\beta,\rho^\beta)
		\in B.
	\end{align*}
	We claim that $(\mu^t,\nu^t,\rho^t)\in A^\alpha$
	holds for all $t\in[0,1/2]$.
	Indeed, the condition
	$\nu^t\leq 0$ a.e.\ on $\alpha$ follows from
	$\nu^\beta\leq0$ a.e.\ on $\beta\subset\alpha$
	and $\nu^\beta\leq c_0$ a.e.\ on
	$\alpha\setminus\beta=\set{\hat\nu^\alpha< -c_0}$,
	and the condition
	$\mu^t\leq 0$ a.e.\ on $\Omega\setminus\alpha$ follows from
	$\mu^\beta\leq0$ a.e.\ on
	$\Omega\setminus\alpha\subset\Omega\setminus\beta$.

	We are interested in lower bounds on $\nu^t$.
	Combining \eqref{eq:mu_nu_leq} and \eqref{eq:partial_boundedness} yields
	the estimate $\nu^\beta\geq -\mu^\beta-c_0\geq -c_0$
	a.e.\ on $\Omega\setminus\beta$.
	Now we have
	\begin{align*}
		\nu^t = (1-t)\hat\nu^\alpha + t\nu^\beta
		&\geq (1-t) (-c_3(\alpha)c_0) -t c_0
		&&\aeon \Omega\setminus\beta,
		\\
		\nu^t = (1-t)\hat\nu^\alpha + t\nu^\beta
		&\geq (1-t)(-c_0) + t(-c_1(\beta)c_0)
		&&\aeon \beta,
		\\
		\nu^t  = (1-t)\hat\nu^\alpha + t\nu^\beta
		&\leq (1-t)c_0 + tc_0 = c_0
		&&\aeon\Omega.
	\end{align*}
	This implies
	\begin{equation*}
		\abs{\nu^t} \leq 
		\max\paren[\big]{
		(1-t)c_3(\alpha)+t,1-t+tc_1(\beta),1}c_0
		\qquad\aeon\Omega.
	\end{equation*}
	We set $\hat t\in (0,1/2)$ such that 
	$(1-\hat t)c_3(\alpha)+\hat t>1-\hat t+\hat tc_1(\beta)$
	holds (which is possible due to $c_3(\alpha)>1$).
	This results in
	$ \max((1-\hat t)c_3(\alpha)+\hat t,1)< c_3(\alpha)$ and
	\begin{equation*}
		\abs{\nu^{\hat t}}
		\leq \max\paren[\big]{(1-\hat t)c_3(\alpha)+\hat t,1}c_0
		< c_3(\alpha)c_0
		\quad\aeon\Omega.
	\end{equation*}
	Due to $(\mu^{\hat t},\nu^{\hat t},\rho^{\hat t})\in B\cap A^\alpha$
	this is a contradiction to \eqref{eq:c_3_def}.
	Thus, our assumption $c_3(\alpha)>1$ is false.
	Using \eqref{eq:hat_mu_alpha_intro} results
	in $\abs{\hat\nu^\alpha}\leq c_0$ a.e.\ on $\Omega$.
	Then we also have
	\begin{equation*}
		\max(\hat\nu^\alpha,\hat\mu^\alpha)\leq 2c_0
		\aeon\Omega
	\end{equation*}
	and
	\begin{equation*}
		\norm{\hat\nu^\alpha}_{L^2(\Omega)}
		+\norm{\hat\mu^\alpha}_{L^2(\Omega)}
		\leq 3\norm{c_0}_{L^2(\Omega)}.
	\end{equation*}
	Since the bounds no longer depend on $\alpha$,
	we can apply \cref{thm:schinabeck},
	which yields a point
	$(\bar\mu,\bar\nu)\in\clconv\set{(\hat\mu^\alpha,\hat\nu^\alpha)\given\alpha\in\AA}$
	such that $\bar\mu$, $\bar\nu$ satisfy \eqref{eq:mstat}.
	By the definition of $\clconv$, there exists a sequence
	$\set{(\mu_k,\nu_k)}_{k\in\N}\subset
	\conv\set{(\hat\mu^\alpha,\hat\nu^\alpha)\given\alpha\in\AA}$
	such that $(\mu_k,\nu_k)\to(\bar\mu,\bar\nu)$.
	Let $\set{\rho_k}_{k\in\N}\subset Y$ be a sequence such that
	$(\mu_k,\nu_k,\rho_k)\in
	\conv\set{(\hat\mu^\alpha,\hat\nu^\alpha,\hat\rho^\alpha)\given\alpha\in\AA} \subset B$
	holds for all $k\in\N$.
	Due to \eqref{eq:B_rho_estimate},
	the sequence $\set{\rho_k}_{k\in\N}$ is a bounded sequence
	and, thus, possesses a weak accumulation point $\bar\rho$.
	Since $B$ is weakly sequentially closed,
	this yields $(\bar\mu,\bar\nu,\bar\rho) \in B$,
	which completes the proof.
\end{proof}

\section{A class of linear MPCCs in Lebesgue spaces}
\label{sec:linear_mpcc}
We consider the linear MPCC
\begin{equation*}
	\label{eq:mpcc2}
	\tag{MPCC$_{\text{lin}}$}
	\begin{minproblem}[u , w,\xi\in L^2(\Omega)]{%
			\dual{F_u}{u} + \dual{F_w}{w} + \dual{F_\xi}{\xi}
		}
		A u -w - \xi &= 0,
		\\
		u&= 0\quad\aeon\Omega^{0+},
		\\
		0 \leq u \perp \xi &\geq 0\quad\aeon\Omega^{00},
		\\
		\xi &= 0\quad\aeon\Omega^{+0},
		\\
		w&\geq 0 \quad\aeon\Omega_w,
	\end{minproblem}
\end{equation*}
where $A \colon L^2(\Omega)\to L^2(\Omega)$
is a linear and continuous operator,
$(\Omega,\AA,m)$ is a measure space,
and $F_u, F_w,F_\xi\in L^{2}(\Omega)$ are fixed.
The sets $\Omega^{0+},\Omega^{00},\Omega^{+0}$
are a measurable partition of $\Omega$,
and $\Omega_w\subset\Omega$ is measurable.

Since \eqref{eq:mpcc2}
will appear as a linearized problem,
see \cref{lem:linearized_OC} below,
we are only interested
in the feasible point $(0,0,0)$.
In the case that $(0,0,0)$ is a minimizer, we want to show that
M-stationarity holds (under some additional assumptions on the structure
of \eqref{eq:mpcc2}).
As we are going to apply the methods in \cref{sec:a_to_m},
we first need to show that the multipliers to the system of
A$_\beta$-stationarity exist for all measurable $\beta\subset\Omega^{00}$
and that they satisfy the regularity requirements.

\subsection{A$_\forall$-stationarity}
\label{sec:astat}

In this section we show
that A$_\forall$-stationarity
is satisfied.
To this end,
let $\beta\subset\Omega^{00}$ be measurable.
The A$_\beta$-stationary multipliers
will be constructed
via the optimality system of
the tightened linear optimization problem
\begin{equation*}
	\label{eq:lpbeta}
	\tag{$\textup{LP}(\beta)$}
	\begin{minproblem}[u , w,\xi\in L^2(\Omega)]{%
			\dual{F_u}{u} + \dual{F_w}{w} + \dual{F_\xi}{\xi}
		}
		Au &-w - \xi = 0,
		\\
		u &\geq0 \quad\aeon \Omega^{00}\setminus\beta,
		&
		u&= 0\quad\aeon\Omega^{0+}\cup\beta,
		\\
		\xi &\geq 0 \quad\aeon \beta,
		&
		\xi &= 0\quad\aeon\Omega^{+0}\cup (\Omega^{00}\setminus\beta),
		\\
		w &\geq 0 \quad\aeon \Omega_w.
	\end{minproblem}
\end{equation*}
Note that if $(0,0,0)$ is a minimizer of \eqref{eq:mpcc2}, then it is also
a minimizer of \eqref{eq:lpbeta}.
In the sequel,
we show that $(0,0,0)$ is a KKT point of \eqref{eq:lpbeta}
if it is a minimizer.
Although the problem \eqref{eq:lpbeta} is linear and seems to be quite innocent,
this can be quite complicated
and only works for some instances of \eqref{eq:mpcc2}.
We provide a counterexample (\cref{ex:no_fcq})
which fails to satisfy the KKT conditions
at the local minimizer $(0,0,0)$ of \eqref{eq:lpbeta}.

The KKT system for \eqref{eq:lpbeta} at $(0,0,0)$ is
given by
\begin{equation*}
	\label{eq:kktbeta}
	\tag{$\textup{KKT}(\beta)$}
	\begin{aligned}
		F_u &+ A\adjoint p + \mu = 0,
		&
		F_w &- p + \lambda = 0,
		&
		F_\xi - p + \nu &= 0,
		\\
		\mu&\leq0 \quad\aeon\Omega^{00}\setminus\beta,
		&
		\mu&=0 \quad\aeon\Omega^{+0},
		\\
		\nu&\leq0 \quad\aeon\beta,
		&
		\nu&=0 \quad\aeon\Omega^{0+},
		\\
		\lambda &\leq 0 \quad\aeon\Omega_w,
		&
		\lambda &= 0 \quad\aeon\Omega\setminus\Omega_w,
	\end{aligned}
\end{equation*}
where $p,\nu,\lambda,\mu\in L^2(\Omega)$ are multipliers.
Note that this system coincides with the system of 
A$_\beta$-stationarity for \eqref{eq:mpcc2}.

We formulate scenarios under which we are
able to prove A$_\forall$-stationarity.
\begin{assumption}
	\label{asm:operator_A}
	\begin{enumerate}
		\item
			\label{asm:operator_A:nonneg}
			The operator $A$ satisfies $A\adjoint v\geq0$ a.e.\ on $\Omega$
			for all $v\in L^2(\Omega)$ with $v\geq0$ a.e.\ on $\Omega$,
			and $m(\Omega^{+0})=0$ holds.
		\item
			\label{asm:operator_A:average}
			The set $\Omega$ has finite measure and
			the operator $A$ is given by
			$Av := v + \dual{1}{v}$ for all $v\in L^2(\Omega)$.
		\item
			\label{asm:operator_A:average_with_scalars}
			The set $\Omega$ has finite measure and
			the operator $A$ is given by
			$Av := d_1v + d_2\dual{1}{v}$ for all $v\in L^2(\Omega)$,
			where $d_1,d_2>0$ are fixed.
	\end{enumerate}
\end{assumption}
Note that the operator given in part~\ref{asm:operator_A:average}
also satisfies the nonnegativity condition for $A\adjoint$ in 
part~\ref{asm:operator_A:nonneg}.
Clearly, part~\ref{asm:operator_A:average} is only a special case of
part~\ref{asm:operator_A:average_with_scalars}.

In the following, we define the function
\begin{equation}
	\label{eq:c_0}
	c_0 := 
	A\adjoint\paren[\big]{\abs{F_u}+\abs{F_w}+\abs{F_\xi}}
	+\abs{F_u}+\abs{F_w}+\abs{F_\xi}\in L^2(\Omega).
\end{equation}
Under \cref{asm:operator_A}
we have $c_0 \ge 0$ a.e.\ due to the nonnegativity of $A\adjoint$.
This function (together with some constants) 
will be used as a pointwise a.e.\ upper bound on the multipliers in the next
\lcnamecrefs{lem:upper_bound_special_case}.
We show that if multipliers exist, that one can also find multipliers
that can be bounded by $c_0$ (and some constant, which can depend on $\beta$
in the case of \itemref{asm:operator_A:average}).
This serves two purposes:
First, this can be used to show that the set of functions
$F_u,F_w,F_\xi$ for which multipliers exist is closed,
which is sufficient to show the existence of KKT multipliers if $(0,0,0)$ is a minimizer,
see the proof of \cref{prop:kkt_lp_beta}.
Second, the pointwise bounds are required to apply the
theory from \cref{sec:a_to_m}.

Let us start with the easier case of
\itemref{asm:operator_A:nonneg}.
\begin{lemma}
	\label{lem:upper_bound_special_case}
	Suppose that \itemref{asm:operator_A:nonneg} holds.
	Let $p_0,\mu_0,\nu_0,\lambda_0$ be multipliers that satisfy \eqref{eq:kktbeta}.
	Then one can find multipliers $p_1,\mu_1,\nu_1,\lambda_1$
	that satisfy \eqref{eq:kktbeta} such that
	\begin{equation*}
		\max\paren[\big]{\abs{p_1},\abs{\nu_1},\abs{\lambda_1},\abs{\mu_1}}
		\leq
		2c_0
		\quad \aeon \Omega.
	\end{equation*}
\end{lemma}
\begin{proof}
	First, we introduce the function
	\begin{equation*}
		a_0 := \max(p_0,\lambda_0,\nu_0)^- \ge 0.
	\end{equation*}
	Here, $v^- := \max(-v,0)$ in a pointwise sense.
	Then we define the new multipliers via
	\begin{equation*}
		p_1 := p_0 + a_0,
		\quad
		\lambda_1 := \lambda_0 + a_0,
		\quad
		\nu_1 := \nu_0 + a_0,
		\quad
		\mu_1 := \mu_0 - A\adjoint a_0.
	\end{equation*}
	We claim that
	$p_1,\mu_1,\nu_1,\lambda_1$ still satisfies \eqref{eq:kktbeta}.
	Indeed, the first three equations of \eqref{eq:kktbeta}
	are true due to linearity.
	The conditions involving $\mu_1$ are true
	due to $A\adjoint a_0\geq0$ a.e.\ on $\Omega$ and $m(\Omega^{+0})=0$.
	On $\set{\nu_0\leq0}$ we observe that
	$\nu_0\leq -a_0$ and therefore $\nu_0\leq \nu_1\leq0$ holds,
	which implies the conditions involving $\nu_1$.
	Similarly, for the conditions involving $\lambda_1$ we can use
	$\lambda_0\leq\lambda_1\leq0$ on $\set{\lambda_0\leq0}$.

	It remains to verify the pointwise bound.
	We start with bounds for $p_1$.
	The upper bound follows from
	\begin{equation*}
		p_1 = F_w + \lambda_1 \leq F_w
		\quad\aeon\Omega.
	\end{equation*}
	We also have the lower bounds
	\begin{equation*}
		\begin{aligned}
			p_1 &\geq 0 
			\quad&&\aeon\set{p_1\geq0},
			\\
			p_1 &= F_w+\lambda_1 = F_w
			\quad&&\aeon\set{\lambda_1=0},
			\\
			p_1 &= F_{\mrep{\xi}{w}} + \mrep{\nu_1}{\lambda_1} \geq F_\xi
			\quad&&\aeon\set{\nu_1\geq0}.
		\end{aligned}
	\end{equation*}
	Due to the construction of $a_0$,
	we have
	$\Omega= \set{p_1\geq0}\cup\set{\lambda_1=0}\cup\set{\nu_1\geq0}$,
	which leads to the lower bound $p_1\geq \min(F_w,F_\xi,0)$
	a.e.\ on $\Omega$.
	Combining the above shows $\abs{p_1}\leq c_0$ a.e.\ on $\Omega$.
	Then $ \max(\abs{\nu_1},\abs{\lambda_1}) \leq 2c_0$ a.e.\ on $\Omega$ follows.
	Finally, using the sign condition of $A\adjoint$ we get
	\begin{equation*}
		\abs{\mu_1}
		\leq 
		\abs{F_u} + \abs{A\adjoint p_1}
		\leq
		\abs{F_u} + A\adjoint\abs{p_1}
		\leq
		c_0
		\quad\aeon\Omega.
	\end{equation*}
	The desired pointwise a.e.\ inequality follows by taking the maximum.
\end{proof}

The case of
\itemref{asm:operator_A:average}
is more complicated.
We require several \lcnamecrefs{lem:upper_bound_l1}
to investigate different cases.
\begin{lemma}
	\label{lem:upper_bound_l1}
	Suppose that \itemref{asm:operator_A:average} holds.
	Additionally, we assume $m(\Omega^{+0})>0$.
	Then
	\begin{equation*}
		\dual{1}{\abs{p}} \leq \paren[\big]{2+m(\Omega^{+0})^{-1}}c_0
	\end{equation*}
	holds for all $(p,\nu,\mu,\lambda)$ that satisfy \eqref{eq:kktbeta}.
\end{lemma}
\begin{proof}
	First, we have
	$ p=F_w+\lambda \leq F_w $
	a.e.\ on $\Omega$.
	This implies $p^+\leq F_w^+$ a.e.\ on $\Omega$
	and also $\dual{1}{p^+}\leq c_0$,
	where $v^+ := \max(v,0)$.
	From $\mu=0$ a.e.\ on $\Omega^{+0}$ we get
	$ F_u + p + \dual{1}{p} =0$ a.e.\ on $\Omega^{+0}$, or
	\begin{equation*}
		\dual{1}{p} = - p -F_u\geq -F_w-F_u
		\quad\aeon\Omega^{+0}.
	\end{equation*}
	Integrating over $\Omega^{+0}$ yields
	\begin{align*}
		m(\Omega^{+0})\dual{1}{p} 
		&\geq 
		\dual{\chi_{\Omega^{+0}}}{-F_w-F_u}
		\geq
		-\dual{\chi_{\Omega^{+0}}}{\abs{F_w}+\abs{F_u}}
		\\ & \geq 
		-\dual{1}{\abs{F_w}+\abs{F_u}}
		\geq - A\adjoint(\abs{F_w}+\abs{F_u})
		\geq -c_0.
	\end{align*}
	Since $m(\Omega^{+0})>0$ holds, this implies
	$-\dual{1}{p}\leq m(\Omega^{+0})^{-1}c_0$,
	and combined with the estimate for $\dual{1}{p^+}$ this yields
	\begin{equation*}
		\dual{1}{\abs{p}} = 2\dual{1}{p^+} - \dual{1}{p}
		\leq 2c_0 + m(\Omega^{+0})^{-1}c_0.
	\end{equation*}
\end{proof}

\begin{lemma}
	\label{lem:another_special_case}
	Suppose that \itemref{asm:operator_A:average} holds.
	Then we have
	\begin{equation*}
		\abs{p}\leq 2 c_0
		\qquad\aeon\Omega\setminus(\Omega_w\cap\beta)
	\end{equation*}
	for all $(p,\nu,\mu,\lambda)$ that satisfy \eqref{eq:kktbeta}.
	In particular, $\abs{p}\leq 2c_0$ holds a.e.\ on $\Omega$
	if $m(\Omega_w\cap\beta)=0$.
\end{lemma}
\begin{proof}
	Again, we have
	$ p=F_w+\lambda \leq F_w \leq c_0$
	a.e.\ on $\Omega$
	and $\dual{1}{p}\leq c_0$.
	We also have the lower bounds
	\begin{equation*}
		\begin{aligned}
			p = -F_u - \mu - \dual{1}{p} \geq -F_u -c_0
			&\geq -2c_0
			\quad\aeon\Omega^{+0}\cup(\Omega^{00}\setminus\beta),
			\\
			p = F_\xi +\nu = F_\xi
			&\geq -2c_0
			\quad\aeon\Omega^{0+},
			\\
			p = F_w + \lambda = F_w
			&\geq -2c_0
			\quad\aeon\Omega\setminus\Omega_w.
		\end{aligned}
	\end{equation*}
	The result then follows from
	$\Omega^{+0}\cup(\Omega^{00}\setminus\beta)\cup\Omega^{0+}
	\cup(\Omega\setminus\Omega_w)
	=\Omega\setminus(\Omega_w\cap\beta)$.
\end{proof}

\begin{lemma}
	\label{lem:l2_pointwise_bound}
	Suppose that \itemref{asm:operator_A:average} holds.
	Additionally, we assume $m(\Omega^{+0})>0$ and
	$m(\Omega_w\cap\beta)>0$.
	Let $p_0,\mu_0,\nu_0,\lambda_0$ be multipliers that satisfy \eqref{eq:kktbeta}.
	Then one can find multipliers $p_1,\mu_1,\nu_1,\lambda_1$
	satisfying \eqref{eq:kktbeta} such that
	\begin{equation*}
		\abs{p_1}
		\leq
		\paren[\Big]{
			2+m(\Omega_w \cap \beta)^{-1} 
		\paren[\big]{2+m(\Omega^{+0})^{-1}}}c_0
		\qquad\aeon\Omega
		.
	\end{equation*}
\end{lemma}
\begin{proof}
	We define the function
	\begin{equation*}
		a_0 := \chi_{\Omega_w \cap \beta}\max(p_0,\lambda_0,\nu_0)^- \ge 0
	\end{equation*}
	and subsequently define
	\begin{equation*}
		p_2 := p_0 + a_0,
		\quad
		\lambda_2 := \lambda_0 + a_0,
		\quad
		\nu_2 := \nu_0 + a_0.
	\end{equation*}
	We observe that the conditions
	\begin{align*}
		F_w-p_2+\lambda_2 &= 0,
		&
		F_\xi - p_2+\nu_2 &= 0,
		\\
		\nu_2 &\leq 0\quad\aeon\beta,
		&
		\lambda_2 &\leq0 \quad\aeon\Omega
	\end{align*}
	still hold for these new multipliers.
	Next, we want to find a bound on $\abs{p_2}$.
	We already have the bound $\abs{p_2}=\abs{p_0}\leq 2c_0$
	a.e.\ on $\Omega\setminus(\Omega_w\cap\beta)$
	due to \cref{lem:another_special_case}.
	Again, $p_2 = F_w + \lambda_2 \leq F_2\leq c_0$
	yields an upper bound.
	We also have the lower bounds
	\begin{equation*}
		\begin{aligned}
			p_2 &\geq 0 
			\quad&&\aeon\set{p_2\geq0}\cap\Omega_w\cap\beta,
			\\
			p_2 = F_w+\lambda_2 &= F_w
			\quad&&\aeon\set{\lambda_2=0}\cap\Omega_w\cap\beta,
			\\
			p_2 = \nu_2 + F_\xi &\geq F_\xi
			\quad&&\aeon\set{\nu_2\geq0}\cap\Omega_w\cap\beta.
		\end{aligned}
	\end{equation*}
	Using the definition of $a_0$, one can show that
	$\Omega_w\cap\beta\subset 
	\set{p_2\geq0}\cup\set{\lambda_2=0}\cup\set{\nu_2\geq0}$
	holds,
	which leads to the lower bound
	$p_2\geq -c_0$ a.e.\ on $\Omega_w\cap\beta$.
	In summary, we have
	\begin{equation*}
		\abs{p_2}\leq 2 c_0 \quad\aeon\Omega.
	\end{equation*}
	In general, $\mu_2:=\mu_0-A\adjoint a_0$
	does not satisfy $\mu_2=0$ a.e.\ on $\Omega^{+0}$.
	Thus, we need to make further modifications.
	For this purpose, we first define
	\begin{equation*}
		a_2 := -m(\Omega_w \cap \beta)^{-1}\dual{1}{a_0}\chi_{\Omega_w \cap \beta}\leq 0.
	\end{equation*}
	Then we have $ \dual{1}{a_2} = - \dual{1}{a_0}$
	and
	\begin{equation*}
		\begin{aligned}
			\abs{a_2} &= -a_2 \leq
			m(\Omega_w \cap \beta)^{-1}
			\dual{1}{a_0}
			\leq m(\Omega_w \cap \beta)^{-1}\dual{1}{p_0^-}
			\\
			&\leq m(\Omega_w \cap \beta)^{-1}
			\paren[\big]{2+m(\Omega^{+0})^{-1}}c_0,
		\end{aligned}
	\end{equation*}
	where we used \cref{lem:upper_bound_l1} in the last inequality.
	Finally, we define
	\begin{equation*}
		p_1 := p_2 +a_2,
		\quad
		\nu_1 := \nu_2 +a_2,
		\quad
		\lambda_1 := \lambda_2+a_2,
		\quad
		\mu_1 := \mu_0 - A\adjoint a_0 - A\adjoint a_2.
	\end{equation*}
	We claim that these functions satisfy \eqref{eq:kktbeta}.
	That the three equations hold can be seen directly.
	For $\mu_1$ we observe that
	\begin{equation*}
		\mu_1 = \mu_0 - a_0 - a_2 - \dual{1}{a_0 + a_2}
		= \mu_0-a_0-a_2
	\end{equation*}
	holds.
	In particular, we have $\mu_1=\mu_0$ a.e.\ on 
	$\Omega^{+0}\cup(\Omega^{00}\setminus\beta)$
	and thus the conditions for $\mu_1$ in \eqref{eq:kktbeta}
	are satisfied.
	For $\nu_1$, we observe $\nu_1 = \nu_2+a_2 \leq \nu_2\leq 0$
	a.e.\ on $\beta$,
	and because of $a_0 = a_2 = 0$ a.e.\ on $\Omega^{0+}$,
	we also have $\nu_1=0$ a.e.\ on $\Omega^{0+}$.
	Similarly, for $\lambda_1$, we observe 
	$\lambda_1 = \lambda_2+a_2 \leq \lambda_2\leq 0$
	a.e.\ on $\Omega_w$,
	and because of $a_0 = a_2 = 0$ a.e.\ on $\Omega\setminus\Omega_w$
	we also have $\lambda_1=0$ a.e.\ on $\Omega\setminus\Omega_w$.

	To summarize, the multipliers $p_1,\mu_1,\nu_1,\lambda_1$ satisfy
	\eqref{eq:kktbeta}.
	The estimate for $p_1$ follows by combining the estimates
	for $p_2$ and $a_2$.
\end{proof}
Finally, we can make an analogous statement to
\cref{lem:upper_bound_special_case}, but for the case of
\itemref{asm:operator_A:average}.
A notable difference is that the bound depends on $\beta$.
\begin{lemma}
	\label{lem:pointwise_bounds_summary}
	Suppose that \itemref{asm:operator_A:average} holds.
	Let $p_0,\mu_0,\nu_0,\lambda_0$ be multipliers that satisfy \eqref{eq:kktbeta}.
	Then there exist multipliers
	$p_1,\mu_1,\nu_1,\lambda_1$ that satisfy \eqref{eq:kktbeta}
	such that the pointwise a.e.\ inequality
	\begin{equation*}
		\max\paren[\big]{\abs{p_1},\abs{\nu_1},\abs{\lambda_1},\abs{\mu_1}}
		\leq C_\beta c_0
	\end{equation*}
	holds, where the constant $C_\beta >0$ 
	can be chosen as
	$C_\beta = 2$ if $m(\Omega^{+0})=0$,
	$C_\beta = 5 + 2m(\Omega)$ if $m(\Omega_w\cap\beta)=0$,
	and
	$C_\beta =
	1+\paren[\big]{2+m(\Omega_w \cap \beta)^{-1} 
	\paren[\big]{2+m(\Omega^{+0})^{-1}}}(2+m(\Omega)) $
	if $m(\Omega^{+0})>0$ and $m(\Omega_w\cap\beta)>0$.
\end{lemma}
\begin{proof}
	If $m(\Omega^{+0})=0$, then the claim follows from
	\cref{lem:upper_bound_special_case}.
	If $m(\Omega_w\cap\beta)=0$,
	then we choose 
	$(p_1,\mu_1,\nu_1,\lambda_1):= (p_0,\mu_0,\nu_0,\lambda_0)$,
	and \cref{lem:another_special_case} yields a pointwise a.e.\ 
	estimate for $\abs{p_1}$.
	In the remaining case of $m(\Omega^{+0})>0$
	and $m(\Omega_w\cap\beta)>0$ we apply \cref{lem:l2_pointwise_bound},
	which also yields an estimate for $\abs{p_1}$.
	In any case, we have an estimate of the form
	$\abs{p_1}\leq \hat C_\beta c_0$ a.e.\ in $\Omega$,
	where $\hat C_\beta:= 2+m(\Omega_w \cap \beta)^{-1} 
	\paren[\big]{2+m(\Omega^{+0})^{-1}}$
	if $m(\Omega_w\cap\beta)>0$ or
	$\hat C_\beta := 2$ if $m(\Omega_w\cap\beta)=0$.
	From \eqref{eq:kktbeta}, we can obtain the estimates
	$\max(\abs{\nu_1},\abs{\lambda_1})\leq (\hat C_\beta+1)c_0$
	a.e.\ in $\Omega$.
	Finally, we have
	\begin{equation*}
		\begin{aligned}
			\abs{\mu_1}
			&= \abs{-F_u-A\adjoint p_1}
			\leq \abs{F_u}+\abs{A\adjoint p_1}
			\leq c_0 + \abs{p_1} + \abs{\dual{1}{p_1}}
			\\ &\leq 
			(1+\hat C_\beta) c_0 + \dual{1}{\abs{p_1}}
			\leq c_0 + \hat C_\beta c_0+ \hat C_\beta \dual{1}{c_0})
			\leq c_0 + \hat C_\beta c_0 + \hat C_\beta(2+m(\Omega)) c_0
			\\ & = 
			(1+\hat C_\beta(2+m(\Omega))) c_0,
		\end{aligned}
	\end{equation*}
	where the estimate $\dual{1}{c_0}\leq (2+m(\Omega))c_0$
	that was used in the last inequality
	can be obtained from the definition of $c_0$.
\end{proof}

These \lcnamecrefs{lem:upper_bound_special_case}
enable us to verify the existence of KKT multipliers
if $(0,0,0)$ is a minimizer of \eqref{eq:lpbeta}.
\begin{proposition}
	\label{prop:kkt_lp_beta}
	Suppose that \itemref{asm:operator_A:nonneg} 
	or \itemref{asm:operator_A:average}
	holds.
	Further suppose that $(0,0,0)$ is a
	minimizer of \eqref{eq:lpbeta}.
	Then there exist multipliers
	$(p,\mu,\nu,\lambda)$ that satisfy \eqref{eq:kktbeta}.
\end{proposition}
\begin{proof}
	We introduce the notation
	$T \colon L^2(\Omega)^3 \to L^2(\Omega)$,
	$T (u, w, \xi) := A u - w - \xi$,
	\begin{equation*}
		P := \set*{
			(u, w, \xi) \in L^2(\Omega)
			\given
			\begin{aligned}
				u   &\geq0 \;\aeon \Omega^{00}\setminus\beta,
				&
				u   &= 0   \;\aeon\Omega^{0+}\cup\beta,
				\\
				w   &\geq0 \;\aeon \Omega_w,
				\\
				\xi &\geq0 \;\aeon \beta,
				&
				\xi &= 0   \;\aeon\Omega^{+0}\cup (\Omega^{00}\setminus\beta)
			\end{aligned}
		}
	\end{equation*}
	and
	$F := (F_u, F_w, F_\xi) \in L^2(\Omega)^3$.
	Then,
	\eqref{eq:lpbeta}
	reads
	\begin{equation*}
		\text{Minimize} \quad \dual{F}{x} \quad
		\text{such that} \quad T x = 0,\; x \in P
		.
	\end{equation*}
	The optimality of
	$x = (0,0,0)$
	yields
	\begin{equation*}
		-F \in ( T^{-1} (\set{0}) \cap P )\polar,
	\end{equation*}
	while
	\eqref{eq:kktbeta}
	reads
	\begin{equation*}
		-F \in T\adjoint L^2(\Omega) + P\polar.
	\end{equation*}
	Thus, it remains to show that the right-hand sides are equal
	and this holds
	if (and only if)
	the set
	\begin{equation*}
		B
		:=
		T\adjoint L^2(\Omega) + P\polar
		=
		\set*{
			\begin{pmatrix}
				A\adjoint p + \mu
				\\
				-p + \lambda
				\\
				-p+\nu
			\end{pmatrix}
			\given
			\begin{aligned}
				\mu&\leq0 \;\aeon\Omega^{00}\setminus\beta,
				&
				\mu&=0 \;\aeon\Omega^{+0},
				\\
				\lambda &\leq 0 \;\aeon\Omega_w,
				&
				\lambda &= 0 \;\aeon\Omega\setminus\Omega_w,
				\\
				\nu&\leq0 \;\aeon\beta,
				&
				\nu&=0 \;\aeon\Omega^{0+}
			\end{aligned}
		}
	\end{equation*}
	is closed, see \cite[Lemma~2.4.1]{Schirotzek2007}.

	Let $\set{F_k}_k = \set{(F_{u,k}, F_{w,k}, F_{\xi,k})}_k\subset -B\subset L^2(\Omega)^3$ be a sequence
	with $F_k\to F_0$ in $L^2(\Omega)^3$.
	Hence,
	there exists a sequence $\set{(p_k, \lambda_k, \nu_k, \mu_k)}_k \subset L^2(\Omega) \times P\polar$
	such that
	$(p_k, \lambda_k, \nu_k, \mu_k)$
	satisfies
	\eqref{eq:kktbeta}
	with $(F_u, F_w, F_\xi)$ replaced by
	$(F_{u,k}, F_{w,k}, F_{\xi,k})$.
	Due to \cref{asm:operator_A},
	we can apply \cref{lem:upper_bound_special_case} or \cref{lem:pointwise_bounds_summary}
	to
	obtain
	$(\hat p_k, \hat \lambda_k, \hat \nu_k, \hat \mu_k) \in L^2(\Omega) \times P\polar$
	with
	\begin{equation}
		\label{eq:with_hats}
		F_k + T\adjoint \hat p_k + (\hat \lambda_k, \hat \nu_k, \hat \mu_k) = 0
	\end{equation}
	and
	\(
		\max\paren[\big]{\abs{\hat p_k},\abs{\hat \nu_k},\abs{\hat \lambda_k},\abs{\hat \mu_k}}
		\leq C_\beta c_{0,k},
	\)
	where $c_{0,k}$ is defined analogously to \eqref{eq:c_0}.
	Note that $c_{0,k}$ is bounded in $L^2(\Omega)$
	due to the boundedness of $\set{F_k}_k$ in $L^2(\Omega)^3$.
	Therefore,
	the sequence
	$\set{(\hat p_k, \hat \lambda_k, \hat \nu_k, \hat \mu_k)}_k$
	has
	a weakly convergent subsequence.
	By passing to the limit (along this subsequence) in \eqref{eq:with_hats},
	it follows that $F_0\in -B$ due to the weak closeness of $P\polar$.
	Thus, $B$ is closed
	and this yields the claim.
\end{proof}

Finally,
we give an example to show 
that KKT multipliers for the local minimizer $(0,0,0)$ of
\eqref{eq:lpbeta} may fail to exist.
\begin{example}
	\label{ex:no_fcq}
	We consider the case that $\Omega=(-1,1)$
	and
	$S := (-\Delta_0)^{-1}  \colon  L^2(\Omega) \to L^2(\Omega)$, where $\Delta_0$ is the 
	Laplacian equipped with homogeneous Dirichlet boundary conditions.
	We further choose
	\begin{align*}
		A &:= \id + S\adjoint S  \colon  L^2(\Omega)\to L^2(\Omega),
		\\
		\Omega_1 &:= \bigcup_{k\in\N} [2^{-2k-1},2^{-2k}],
		\\
		\Omega_2 &:= \Omega\setminus\Omega_1
		= (-1,0]\cup\bigcup_{k\in\N} (2^{-2k},2^{-2k+1}) \cup [1/2,1),
		\\
		\Omega_w &:= \Omega_1,
		\quad
		\Omega^{+0} := \Omega_2,
		\quad
		\Omega^{00} := \beta := \Omega_1,
		\quad
		\Omega^{0+} := \emptyset,
		\\
		F_w &:= F_\xi := 0,
		\qquad
		F_u(\omega) :=
		\chi_{\Omega_2} (\tfrac16 \abs{\omega}^3-\tfrac12\omega^2 + \tfrac13)
		\qquad\forall \omega\in\Omega.
	\end{align*}
	Then $(0,0,0)$ is a (global) minimizer of \eqref{eq:lpbeta},
	but no multipliers exist such that \eqref{eq:kktbeta} is satisfied.

	Moreover, the set of functions $F_u,F_w,F_\xi$ for which
	multipliers exist such that \eqref{eq:kktbeta} is satisfied,
	is nonempty and not closed.
\end{example}
\begin{proof}
	We start with showing that the set of functions $F_u,F_w,F_\xi$
	for which multipliers exist
	such that \eqref{eq:kktbeta} is satisfied is nonempty and not closed.
	For $k\in\N$, we define
	\begin{align*}
		p_k &:= -2^{2k+2}\chi_{(2^{-2k-1},2^{-2k})},
		\\
		\lambda_k &:= \nu_k := p_k,
		\\
		\mu_k &:= \chi_{\Omega_1}(-A\adjoint p_k),
		\\
		F_u^k &:= \chi_{\Omega_2}(-A\adjoint p_k).
	\end{align*}
	Then one can check that \eqref{eq:kktbeta} is satisfied
	(by using $p_k,\lambda_k,\nu_k,\mu_k,F_u^k$ instead of 
	$p,\lambda,\nu,\mu,F_u$).
	One can also calculate $Sp_k$ and confirm that it satisfies
	\begin{equation*}
		v_k(\omega):=(Sp_k)(\omega) =
		\begin{cases}
			c_k(\omega+1)
			&\text{if}\;
			\omega\in (-1,2^{-2k-1}],
			\\
			c_k+2^{-2k-1} + (c_k-2)\omega
			+ 2^{2k+1}\omega^2
			&\text{if}\;
			\omega\in (2^{-2k-1},2^{-2k}),
			\\
			(c_k+2)(\omega - 1)
			&\text{if}\;
			\omega\in [2^{-2k},1),
		\end{cases}
	\end{equation*}
	where $c_k = -1 + 3\cdot 2^{-2k-2}$.
	We also observe that $v_k$ converges in $L^2(\Omega)$
	with limit $v_0(\omega):=\abs{\omega}-1$,
	and that $(S\adjoint v_0)(\omega)=(Sv_0)(\omega)
	=-(\tfrac16 \abs{\omega}^3-\tfrac12\omega^2 + \tfrac13)$ holds.
	Thus, we have the convergence
	\begin{equation*}
		\begin{aligned}
			F_u^k &= \chi_{\Omega_2}(-A\adjoint p_k)
			= \chi_{\Omega_2}(-p_k-S\adjoint S p_k)
			= \chi_{\Omega_2}(-S\adjoint S p_k)
			= \chi_{\Omega_2}(-S\adjoint v_k)
			\\
			&\to \chi_{\Omega_2}(-S\adjoint v_0)
			= F_u
			\qquad \text{as }k\to\infty
		\end{aligned}
	\end{equation*}
	in $L^2(\Omega)$.
	Suppose that $p,\mu,\nu,\lambda\in L^2(\Omega)$ satisfy
	\eqref{eq:kktbeta}.
	Then we have $p=\lambda=\nu$,
	$p=0$ a.e.\ on $\Omega_2$,
	and $F_u + A\adjoint p = -\mu = 0$ a.e.\ on $\Omega_2$.
	Thus,
	\begin{equation*}
		0 = F_u + A\adjoint p
		= -S\adjoint v_0 + S \adjoint S p
		= S\adjoint(Sp-v_0)
		\qquad\aeon\Omega_2.
	\end{equation*}
	Since $S p$ and $v_0$ are continuous, taking second derivatives
	in $\Omega_2$ in the above equation yields
	$S p = v_0$ on $\Omega_2$.
	We also have equality of derivatives of $S p$ and $v_0$
	in the interior of $\Omega_2$.
	This shows that the derivative of $S p$ is discontinuous at $0$,
	which is a contradiction to its regularity $(S p)' \in H^1(\Omega) \embeds C(\bar\Omega)$.
	Thus, there exist no multipliers that satisfy
	\eqref{eq:kktbeta} with our choices for $F_u$, $F_w$, $F_\xi$.
	Due to $F_u^k\to F_u$, the set of functions
	$F_u,F_w,F_\xi$ for which multipliers exist is nonempty, but not closed.

	It remains to show that $(0,0,0)$ is a minimizer of \eqref{eq:lpbeta}.
	Let $(u,w,\xi)\in L^2(\Omega)^3$ be another feasible point of
	\eqref{eq:lpbeta}.
	Note that $u=0$ a.e.\ on $\Omega_1$
	and $Au = w + \xi \geq0$ a.e.\ on $\Omega_1$.
	For the objective function we have
	\begin{equation*}
		\begin{aligned}
			\dual{F_u}{u} + \dual{F_w}{w} + \dual{F_\xi}{\xi}
			&= \dual{F_u}{u}
			= \lim_{k\to\infty} \dual{F_u^k}{u}
			= \lim_{k\to\infty} \dual{\chi_{\Omega_2}(-A\adjoint p_k)}{u}
			\\
			&= \lim_{k\to\infty} \dual{- p_k}{A \chi_{\Omega_2}u}
			= \lim_{k\to\infty} \dual{- \chi_{\Omega_1}p_k}{\chi_{\Omega_1}A u}
			\\
			& \geq0
			=\dual{F_u}0 + \dual{F_w}0 + \dual{F_\xi}0
		\end{aligned}
	\end{equation*}
	where we used $-p_k\geq0$ and $\chi_{\Omega_1}A u\geq0$
	a.e.\ on $\Omega$ in the last step.
	Since $(u,w,\xi)\in L^2(\Omega)^3$ was an arbitrary feasible point
	of \eqref{eq:lpbeta}, $(0,0,0)$ is a (global) minimizer of this
	optimization problem.
\end{proof}

\subsection{M-stationarity}

We combine the results of
\cref{sec:a_to_m} with those of \cref{sec:astat} to show that
$(0,0,0)$ is M-stationary if it is a minimizer of \eqref{eq:mpcc2}.
Using \cref{def:stat},
the system of M-stationarity for \eqref{eq:mpcc2} at $(0,0,0)$ is
\begin{equation}
	\label{eq:lin_mstat}
	\begin{aligned}
		F_u &+ A\adjoint \bar p + \bar\mu = 0
		&
		F_w &- \bar p + \bar\lambda = 0
		&
		F_\xi &- \bar p + \bar\nu = 0,
		\\
		\bar\lambda &\leq 0 \quad\aeon\Omega_w,
		&
		\bar\lambda &= 0 \quad\aeon\Omega\setminus\Omega_w,
		\\
		\bar\mu&=0 \quad\aeon\Omega^{+0},
		&
		\bar\nu&=0 \quad\aeon\Omega^{0+},
		\\&&
		\mrep[r]{(\bar\mu<0\land\bar\nu<0)\lor\bar\mu\bar\nu}{}&=0
		\quad\aeon\Omega^{00},
	\end{aligned}
\end{equation}
where $\bar p,\bar\nu,\bar\lambda,\bar\mu\in L^2(\Omega)$ are the multipliers.
First, we check
that M-stationarity holds under \itemref{asm:operator_A:nonneg}
for a local minimizer $(0,0,0)$.
\begin{theorem}
	\label{thm:linear_mstat_a}
	Suppose that \itemref{asm:operator_A:nonneg} holds.
	If $(0,0,0)$ is a (global)
	minimizer of \eqref{eq:mpcc2},
	then it is an M-stationary point.
\end{theorem}
\begin{proof}
	For each measurable $\beta\subset\Omega^{00}$,
	$(0,0,0)$ is also a minimizer of \eqref{eq:lpbeta}
	because this problem has a smaller feasible set.
	By \cref{prop:kkt_lp_beta} there exist
	multipliers $(p^\beta,\mu^\beta,\nu^\beta,\lambda^\beta) \in L^2(\Omega)^4$ 
	that satisfy \eqref{eq:kktbeta}.

	We want to apply \cref{thm:schinabeck}
	on $\Omega^{00}$ (instead of $\Omega$).
	By \cref{lem:upper_bound_special_case}, we can choose
	$(p^\beta,\mu^\beta,\nu^\beta,\lambda^\beta) \in L^2(\Omega)^4$
	in such a way that there is a common $L^2(\Omega)$-function
	that bounds these multipliers.
	This implies that \eqref{eq:bounds_on_mu_nu_1} and \eqref{eq:bounds_on_mu_nu_2}
	are satisfied.
	Applying \cref{thm:schinabeck} yields a point
	$(\bar\mu^{00},\bar\nu^{00})\in 
	\clconv\set{(\chi_{\Omega^{00}}\mu^\beta,\chi_{\Omega^{00}}\nu^\beta)
	\given \beta\subset\Omega^{00}\;\text{measurable}}$
	which satisfies
	$(\bar\mu^{00}<0\land\bar\nu^{00}<0)\lor\bar\mu^{00}\bar\nu^{00}=0$
	a.e.\ on $\Omega^{00}$.
	By the definition of $\clconv$, there exists a sequence
	$\set{(\mu^{00}_k,\nu^{00}_k)}_{k\in\N}\subset\conv\set{
		(\chi_{\Omega^{00}}\mu^\beta,\chi_{\Omega^{00}}\nu^\beta)
		\given \beta\subset\Omega^{00}\;\text{measurable}
	}$
	with $\mu^{00}_k\to\bar\mu^{00}$, $\nu^{00}_k\to\bar\nu^{00}$.
	Using the same convex combination, one can find
	multipliers
	$(p_k,\mu_k,\nu_k,\lambda_k)\in
	\conv\set{
		(p^\beta,\mu^\beta,\nu^\beta,\lambda^\beta) 
		\given \beta\subset\Omega^{00}\;\text{measurable}
	}
	$
	with $\mu^{00}_k=\chi_{\Omega^{00}}\mu_k$
	and $\nu^{00}_k=\chi_{\Omega^{00}}\nu_k$
	for all $k\in\N$.
	Due to convexity, these multipliers satisfy the conditions of
	\eqref{eq:kktbeta} that do not depend on $\beta$.
	As $(p_k,\mu_k,\nu_k,\lambda_k)$ is bounded,
	there exists a weakly convergent subsequence whose weak limit we will denote by
	$(\bar p,\bar\mu,\bar\nu,\bar\lambda)$.
	These multipliers also satisfy the conditions of
	\eqref{eq:kktbeta} that do not depend on $\beta$.
	It remains to show that the last condition of \eqref{eq:lin_mstat}
	holds for $\bar\mu$ and $\bar\nu$.
	This, however, is true due to $\bar\mu=\bar\mu^{00}$ and $\bar\nu=\bar\nu^{00}$
	a.e.\ on $\Omega^{00}$.
\end{proof}
For the case of \itemref{asm:operator_A:average}, the
proof is a bit more complicated,
as it is based on \cref{lem:preprocessing},
which has other requirements than \cref{thm:schinabeck}.
\begin{theorem}
	\label{thm:linear_mstat_b}
	Suppose that \itemref{asm:operator_A:average} or
	\itemref{asm:operator_A:average_with_scalars}
	holds.
	If $(0,0,0)$ is a (global)
	minimizer of \eqref{eq:mpcc2},
	then it is an M-stationary point.
\end{theorem}
\begin{proof}
	We start with the case that
	\itemref{asm:operator_A:average}
	holds.
	For each measurable $\beta\subset\Omega^{00}$,
	$(0,0,0)$ is also a minimizer of \eqref{eq:lpbeta}
	because it has a smaller feasible set.
	By \cref{prop:kkt_lp_beta} there exist
	multipliers $(p^\beta,\mu^\beta,\nu^\beta,\lambda^\beta)
	\in L^2(\Omega)^4$ that satisfy \eqref{eq:kktbeta}.

	Let us argue that it is possible to choose the
	multipliers $(p^\beta,\mu^\beta,\nu^\beta,\lambda^\beta)$
	in such a way that they satisfy \eqref{eq:kktbeta} and
	\begin{align}
		\label{eq:avg_abs_est}
		\dual{1}{\abs{p^\beta}} &\leq C c_0
		\\
		\label{eq:beta_dependent_nu_estimate}
		\abs{\nu^\beta} &\leq C_\beta c_0
	\end{align}
	hold a.e.\ in $\Omega$, where $C>0$ is a constant which
	does not depend on $\beta$ and $C_\beta>0$ is a constant
	which can depend on $\beta$.

	Indeed, if $\Omega^{+0}$ has measure zero,
	then we can use \cref{lem:upper_bound_special_case} to redefine
	the multipliers $(p^\beta,\mu^\beta,\nu^\beta,\lambda^\beta)$
	in such a way that $\max(\abs{p^\beta},\abs{\nu^\beta})\leq 2c_0$
	is satisfied.
	In this case \eqref{eq:avg_abs_est} is satisfied
	with $C=2m(\Omega)$ and \eqref{eq:beta_dependent_nu_estimate}
	is satisfied with $C_\beta=2$.
	If, on the other hand, $\Omega^{+0}$ has positive measure,
	then we can use \cref{lem:pointwise_bounds_summary} to redefine
	the multipliers $(p^\beta,\mu^\beta,\nu^\beta,\lambda^\beta)$
	such that \eqref{eq:beta_dependent_nu_estimate} is satisfied.
	In this case, \eqref{eq:avg_abs_est} follows from 
	\cref{lem:upper_bound_l1}.

	We want to apply \cref{lem:preprocessing}
	on $\Omega^{00}$ (instead of $\Omega$),
	where we consider the restrictions of $\mu^\beta,\nu^\beta$
	on $\Omega^{00}$,
	and with
	\begin{equation*}
		\rho^\beta := (p^\beta,\chi_{\Omega^{0+}}\mu^\beta,
		\chi_{\Omega^{+0}}\nu^\beta,\lambda^\beta)
		\in L^2(\Omega)\times L^2(\Omega^{0+})
		\times L^2(\Omega^{+0})\times L^2(\Omega)
		=: Y.
	\end{equation*}
	The conditions \eqref{eq:mu_nu_leq}
	follow directly from \eqref{eq:kktbeta}.
	Next, we will show
	\begin{equation}
		\label{eq:boundedness_mu_nu_asm}
		\max(\abs{\mu^\beta+\nu^\beta},\nu^\beta)
		\leq Cc_0
	\end{equation}
	for all measurable $\beta\subset\Omega^{00}$,
	where $C>0$ is a constant which does not depend on $\beta$.
	From \eqref{eq:kktbeta} we obtain
	\begin{equation*}
		\nu^\beta = -F_\xi + p^\beta
		= -F_\xi + F_w + \lambda^\beta
		\leq -F_\xi + F_w \leq c_0
		\quad\aeon\Omega.
	\end{equation*}
	From \eqref{eq:kktbeta} we also obtain
	\begin{equation*}
		\mu^\beta+\nu^\beta
		= - F_\xi + p^\beta - F_u - A\adjoint p^\beta
		= -F_\xi + - F_u - \dual{1}{p^\beta}.
	\end{equation*}
	This implies
	$\abs{\mu^\beta+\nu^\beta}
	\leq c_0 + \dual{1}{\abs{p^\beta}}\leq (1+C)c_0$
	a.e.\ in $\Omega$ due to \eqref{eq:avg_abs_est}.
	In summary, \eqref{eq:boundedness_mu_nu_asm} holds.
	The assumptions \eqref{eq:partial_boundedness} and \eqref{eq:alpha_dependent_bound}
	follow from \eqref{eq:beta_dependent_nu_estimate}
	and \eqref{eq:boundedness_mu_nu_asm} with suitable constants.

	Next, we will show \eqref{eq:rho_bounded} with the bounded operator
	$T\colon L^2(\Omega^{00})\to Y$ given by $T\nu = (\nu,0,0,\nu)$.
	We observe that we have
	$\abs{p^\beta}\leq 2c_0$ a.e.\ on $\Omega\setminus\Omega^{00}$
	by \cref{lem:another_special_case}.
	This can be combined with \eqref{eq:avg_abs_est} to yield
	$\abs{A\adjoint p^\beta}\leq \abs{p^\beta}+\dual{1}{\abs{p^\beta}} \leq 
	(2+C)c_0$ a.e.\ on $\Omega\setminus\Omega^{00}$.
	Using these properties and \eqref{eq:kktbeta} one can obtain
	estimates of the form
	\begin{equation*}
		\norm{p^\beta}_{L^2(\Omega\setminus\Omega^{00})}
		+\norm{\nu^\beta}_{L^2(\Omega\setminus\Omega^{00})}
		+\norm{\mu^\beta}_{L^2(\Omega\setminus\Omega^{00})}
		+\norm{\lambda^\beta}_{L^2(\Omega\setminus\Omega^{00})}
		\leq (11+C)\norm{c_0}_{L^2(\Omega)}.
	\end{equation*}
	Then \eqref{eq:rho_bounded} follows by the estimates
	\begin{align*}
		\norm{\rho^\beta-T(\chi_{\Omega^{00}}\nu^\beta)}_Y
		&= 
		\norm{p^\beta-\chi_{\Omega^{00}}\nu^\beta}_{L^2(\Omega)}
		+\norm{\mu^\beta}_{L^2(\Omega^{0+})}
		\\ &\qquad
		+\norm{\nu^\beta}_{L^2(\Omega^{+0})}
		+\norm{\lambda^\beta-\chi_{\Omega^{00}}\nu^\beta}_{L^2(\Omega)}
		\\ &\leq
		\norm{p^\beta}_{L^2(\Omega\setminus\Omega^{00})}
		+\norm{p^\beta-\nu^\beta}_{L^2(\Omega^{00})}
		+\norm{\mu^\beta}_{L^2(\Omega^{0+})}
		\\ &\qquad
		+\norm{\nu^\beta}_{L^2(\Omega^{+0})}
		+\norm{\lambda^\beta}_{L^2(\Omega\setminus\Omega^{00})}
		+\norm{\lambda^\beta-\nu^\beta}_{L^2(\Omega^{00})}
		\\ &\leq
		(11+C)\norm{c_0}_{L^2(\Omega)}+\norm{F_\xi}_{L^2(\Omega^{00})}
		+
		\norm{p^\beta-F_w-p^\beta+F_\xi}_{L^2(\Omega^{00})}
		\\ &\leq
		(14+C)\norm{c_0}_{L^2(\Omega)}.
	\end{align*}
	In summary, \cref{lem:preprocessing} can be applied.
	Thus, there exists a point
	$(\bar p,\bar\mu,\bar\nu,\bar\lambda)$ in the closure
	of the convex hull of 
	$\set{(p^\beta,\mu^\beta,\nu^\beta,\lambda^\beta)\given
	\beta\subset\Omega^{00}\;\text{measurable}}$
	which satisfies \eqref{eq:mstat}.
	Due to convexity, the other conditions
	of \eqref{eq:lin_mstat} follow from \eqref{eq:kktbeta}.

	For the case of
	\itemref{asm:operator_A:average_with_scalars},
	the result follows from the result for the case of
	\itemref{asm:operator_A:average_with_scalars} by rescaling.
	First, one can scale $A$ (and rescale the data, variables and multipliers accordingly)
	such that the result applies to
	the case that $A$ is given by $Av = d_1v+d_1\dual{1}{v}$
	for all $v\in L^2(\Omega)$.
	Then, one can scale the measure $m$ by $d_2/d_1$
	to allow for an operator $A$ with
	$A v= d_1v + d_2 \dual{1}{v}$ for all $v\in L^2(\Omega)$.
\end{proof}

\section{An inverse optimal control problem}
%%fakesubsection: Intro
\label{sec:IOC}
In this section, we consider a class of inverse optimal control problems.
We will see that our M-stationarity results
from \cref{sec:a_to_m,sec:linear_mpcc}
can be applied to the KKT reformulation of this
inverse optimal control problem under some assumptions.

\subsection{Problem setting and existence of solutions}
\label{subsec:IOC_setting}
Inverse optimal control problems can be modeled
as bilevel optimization problems.
We are interested in the following problem class.
The upper-level problem is given by
\begin{equation*}
	\label{eq:IOC}
	\tag{IOC}
	\begin{minproblem}[u \in L^2(\Omega), w \in H_0^1(\Omega)]
		{
			F(u, w)
			:=
			f(u)
			+ \frac12 \norm{w}_{H_0^1(\Omega)}^2
			+ \dual{\zeta}{w}_{L^2(\Omega)}
		}
		& \text{$u$ solves \eqref{eq:OC_w}}, \\
		& w \ge w_a \quad\aeon\Omega,
	\end{minproblem}
\end{equation*}
while
for each $w \in L^2(\Omega)$
the lower-level problem
is
the parametrized optimal control problem
\begin{equation*}
	\label{eq:OC_w}
	\tag{OC$(w)$}
	\begin{minproblem}[u \in L^2(\Omega)]{
		\frac12 \norm{S u - y_d}_Y^2 + \frac\alpha2 \norm{u - w}_{L^2(\Omega)}^2
	}
	u \ge u_a \quad\aeon\Omega.
	\end{minproblem}
\end{equation*}
We fix the assumptions concerning the data
in both problems.
Throughout this section,
$\Omega \subset \R^n$ is open and bounded
(and equipped with the Lebesgue measure), $n \in \N$,
$Y$ is a (real) Hilbert space, $S \in \LL(L^2(\Omega), Y)$,
$\alpha > 0$,
$u_a \in L^2(\Omega)$.
For the upper-level problem,
we assume
$w_a \in H^1(\Omega)$ such that $-\Delta w_a \in L^2(\Omega)$
and $w_a \le 0$ on $\partial\Omega$ in the sense $\max(w_a, 0) \in H_0^1(\Omega)$,
$\zeta \in L^2(\Omega)$
and $f \colon L^2(\Omega) \to \R$
is assumed to be continuously Fréchet differentiable
and is bounded from below.
We further require some regularity,
i.e.,
we assume that there exists $p > 2$
such that
$S\adjoint S \in \LL(L^2(\Omega), L^p(\Omega))$,
$S\adjoint y_d, u_a \in L^p(\Omega)$
and $H_0^1(\Omega) \embeds L^p(\Omega)$.
Finally, we mention that we use
\begin{equation*}
	\norm{w}_{H_0^1(\Omega)}
	:=
	\norm{\nabla w}_{L^2(\Omega)^n},
\end{equation*}
i.e., we equip $H_0^1(\Omega)$
with the $H^1(\Omega)$-seminorm,
and the associated inner product is denoted by
\begin{equation*}
	\dual{v}{w}_{H_0^1(\Omega)}
	:=
	\dual{\nabla v}{\nabla w}_{L^2(\Omega)^n}
	.
\end{equation*}

We note that the upper-level variable $w$
enters the lower-level problem
\eqref{eq:OC_w}
as the desired control (sometimes also denoted reference control).

We define the lower-level and upper-level feasible sets by
\begin{equation*}
	\Uad := \set{v \in L^2(\Omega) \given v \ge u_a \text{ a.e.}}
	,\qquad
	\Wad := \set{v \in H_0^1(\Omega) \given v \ge w_a \text{ a.e.}}
	,
\end{equation*}
respectively.
It is well known that
for fixed $w \in L^2(\Omega)$,
the unique solution $u \in \Uad$ of \eqref{eq:OC_w}
is characterized by the variational inequality (VI)
\begin{equation}
	\label{eq:OC_w_VI}
	\dual[\big]{S\adjoint(S u - y_d) + \alpha (u - w)}{v - u}_{L^2(\Omega)}
	\ge
	0
	\qquad
	\forall v \in \Uad.
\end{equation}
For later reference,
the solution operator
(mapping $w \mapsto u$)
of
\eqref{eq:OC_w}
(or, equivalently, \eqref{eq:OC_w_VI})
is denoted by
$T \colon L^2(\Omega)\to L^2(\Omega)$.
Using the continuity of $T$
and the compact embedding $H_0^1(\Omega) \embeds L^2(\Omega)$,
see \cite[Theorem~1.34]{Troianiello1987},
we can show the existence of solutions of \eqref{eq:IOC}.

\begin{lemma}
	\label{lem:existence_IOC}
	The inverse optimal control problem \eqref{eq:IOC}
	possesses a global solution.
\end{lemma}
\begin{proof}
	The solution mapping $w \mapsto u$
	of the lower-level problem \eqref{eq:OC_w}
	is continuous from $L^2(\Omega)$ to $L^2(\Omega)$,
	which can be seen from
	the projection formula \eqref{eq:lower-level_projection}.
	Since the objective of \eqref{eq:IOC}
	is coercive w.r.t.\ $w \in H_0^1(\Omega)$,
	existence of solutions follows by
	the direct method of calculus of variations.
\end{proof}
Due to the non-convexity of \eqref{eq:IOC},
one cannot expect uniqueness of solutions.

\subsection{Regularity of solutions}
\label{subsec:IOC_regularity}
In this section,
we show that every locally optimal $\bar w \in H_0^1(\Omega)$
of \eqref{eq:IOC} enjoys the additional regularity
$-\Delta\bar w \in L^2(\Omega)$.
This will become important for a density argument in the B-stationarity system,
see \cref{lem:B_stat} below.

We mention that a similar regularity effect is well known
for the optimal control of the obstacle problem.
Using a regularization approach,
one can derive a system of C-stationarity
and this can be used to show $H^1$-regularity of the control
using the projection formula.
We further note,
that such a regularization procedure is not needed
in the unconstrained setting,
since the additional regularity follows directly from
the system of B-stationarity,
see \cite[Proof of Proposition~4.1]{Mignot1976},
\cite[Proof of Theorem~5.3]{Wachsmuth2016:2}.

We also mention that we do not need to derive
a system of C-stationarity in its full glory,
since the additional regularity of $\bar w$
is sufficient in the next section.

We start by regularizing the VI \eqref{eq:OC_w_VI}.
To this end, we introduce
the monotone penalty function $\pi \colon \R \to \R$ via
\begin{equation*}
	\pi(s) :=
	\begin{cases}
		s + \frac12 & \text{if }s\leq -1, \\
		-\frac12 s^2  & \text{if }-1<s< 0, \\
		0 & \text{if }s\geq 0.
	\end{cases}
\end{equation*}
Note that $\pi$ is continuously differentiable
with a bounded derivative.
The Nemytskii operators associated with $\pi$ and $\pi'$
are denoted by the same symbols, respectively.
We approximate \eqref{eq:OC_w_VI}
by the equation
\begin{equation*}
	\label{eq:OC_ud_gamma}
	\tag{OC$_\gamma(w)$}
	e_\gamma(u, w)
	:=
	S\adjoint (S u - y_d) + \alpha (u - w) + \gamma \pi(u - u_a) = 0
	,
\end{equation*}
where
$\gamma > 0$ is the regularization parameter.
The first lemma provides the differentiability of this regularization.
\begin{lemma}
	\label{lem:regularized_OC}
	For every $w \in L^p(\Omega)$,
	there exists a unique solution $u =: T_\gamma(w) \in L^p(\Omega)$
	of \eqref{eq:OC_ud_gamma}.
	The solution operator $T_\gamma \colon L^p(\Omega) \to L^2(\Omega)$
	is Fréchet differentiable.
	For all $w,h \in L^p(\Omega)$,
	the derivative $v = T_\gamma'(w) h$ satisfies
	\begin{equation}
		\label{eq:derivative_T_gamma}
		S\adjoint S v + \alpha ( v - h ) + \gamma \pi'(u - u_a) v = 0,
	\end{equation}
	where $u = T_\gamma(w)$.
\end{lemma}
\begin{proof}
	By monotone operator theory,
	e.g. \cite[Theorem~4.17]{Troianiello1987},
	the existence of a solution $u \in L^2(\Omega)$
	of \eqref{eq:OC_ud_gamma}
	for all $w \in L^2(\Omega)$ follows.
	Uniqueness and Lipschitz continuity of $T_\gamma \colon L^2(\Omega) \to L^2(\Omega)$
	follows by a standard argument:
	if $u_i$ is a solution for $w_i$, $i = 1, 2$,
	we get
	\begin{align*}
		0
		&=
		\dual{e_\gamma(u_2, w_2) - e_\gamma(u_1, w_1)}{u_2 - u_1}_{L^2(\Omega)}
		\\&\ge
		\alpha \norm{u_2 - u_1}_{L^2(\Omega)}^2
		-
		\alpha \dual{w_2 - w_1}{u_2 - u_1}_{L^2(\Omega)}
	\end{align*}
	and this yields $\norm{u_2 - u_1}_{L^2(\Omega)} \le \norm{w_2 - w_1}_{L^2(\Omega)}$.

	In case $w \in L^p(\Omega)$,
	the additional regularity of $u$ follows by a bootstrapping argument:
	By rearranging the equation, we find
	\begin{equation*}
		u - u_a
		=
		(\alpha \id + \gamma \pi)^{-1} \parens[\big]{
			S\adjoint (y_d - S u) + \alpha (w - u_a)
		},
	\end{equation*}
	where $(\alpha \id + \gamma \pi)^{-1} \colon \R \to \R$
	is the Lipschitz continuous inverse of
	$\alpha \id + \gamma \pi \colon \R \to \R$.
	Due to the assumed regularities
	$S\adjoint S \in \LL(L^2(\Omega), L^p(\Omega))$,
	$S\adjoint y_d, u_a \in L^p(\Omega)$,
	the Lipschitz regularity
	$T_\gamma \colon L^p(\Omega) \to L^2(\Omega)$
	is improved to the spaces
	$T_\gamma \colon L^p(\Omega) \to L^p(\Omega)$.

	Using \cite[Theorem~7]{GoldbergKampowskyTroeltzsch1992},
	we find that
	$e_\gamma \colon L^p(\Omega) \times L^p(\Omega) \to L^2(\Omega)$
	is Fréchet differentiable
	and the partial Fréchet derivative
	$e_{\gamma, u}(u, w) \in \LL(L^p(\Omega), L^2(\Omega))$
	is given by
	\begin{equation*}
		e_{\gamma, u}(u, w) v
		=
		S\adjoint S v + \alpha v + \gamma \pi'(u - u_a) v
		\qquad\forall v \in L^p(\Omega)
		.
	\end{equation*}
	This operator can be extended to
	$e_{\gamma, u}(u, w) \in \LL(L^2(\Omega), L^2(\Omega))$
	and, via the Lemma of Lax--Milgram, this extension is bijective.

	Now, we can apply the differentiability result
	\cite[Theorem~2.1]{Wachsmuth2012:2}
	with the setting
	\begin{equation*}
		Y^+ = U = L^p(\Omega),
		\qquad
		Y^0 = Z^0 = L^2(\Omega).
	\end{equation*}
	This shows the Fréchet differentiability of
	$T_\gamma \colon L^p(\Omega) \to L^2(\Omega)$
	and the stated equation for the derivative.
\end{proof}
Next, we address the convergence of the regularization.
\begin{lemma}
	\label{lem:convergence_reg}
	Let $\seq{w_k}_{k \in \N} \subset L^2(\Omega)$
	be a sequence with $w_k \to w$ in $L^2(\Omega)$.
	Further, let $\seq{\gamma_k}_{k \in \N} \subset \R^+$
	be given such that $\gamma_k \to \infty$.
	We denote by $u_k = T_{\gamma_k}(w_k) \in L^2(\Omega)$
	the corresponding unique solution of
	\eqref{eq:OC_ud_gamma},
	i.e., $e_{\gamma_k}(u_k, w_k) = 0$.
	Then, $u_k \to u$ in $L^2(\Omega)$,
	where $u$ solves \eqref{eq:OC_w}.
\end{lemma}
\begin{proof}
	We start by multiplying
	the equation
	$e_{\gamma_k}(u_k, w_k) = 0$
	by
	$u_k - u_a$.
	This yields
	\begin{align*}
		&\norm{S u_k}_Y^2 + \alpha \norm{u_k}_{L^2(\Omega)}^2
		+
		\gamma_k \dual{\pi(u_k - u_a)}{u_k - u_a}_{L^2(\Omega)}
		\\&\qquad=
		\dual{y_d + S u_a}{S u_k}_{Y}
		-
		\dual{y_d}{S u_a}_{Y}
		+
		\alpha\dual{w_k + u_a}{u_k}_{L^2(\Omega)}
		-
		\alpha\dual{w_k}{u_a}_{L^2(\Omega)}
		.
	\end{align*}
	On the right-hand side, we can use Young's inequality to obtain
	\begin{equation*}
		\frac12 \norm{S u_k}_Y^2 + \frac\alpha2 \norm{u_k}_{L^2(\Omega)}^2
		+
		\gamma_k \dual{\pi(u_k - u_a)}{u_k - u_a}_{L^2(\Omega)}
		\le
		C,
	\end{equation*}
	where the constant $C \ge 0$ is independent of $k \in \N$.
	Thus, there exists a subsequence (denoted by the same symbol)
	with $u_k \weakly \tilde u$ in $L^2(\Omega)$.
	Due to convexity of the mapping
	$v \mapsto \dual{\pi(v)}{v}_{L^2(\Omega)}$,
	we find
	\begin{equation*}
		0
		\le
		\dual{\pi(\tilde u - u_a)}{\tilde u - u_a}_{L^2(\Omega)}
		\le
		\liminf_{k \to \infty}
		\dual{\pi(u_k - u_a)}{u_k - u_a}_{L^2(\Omega)}
		\le
		\liminf_{k \to \infty}
		\frac{C}{\gamma_k}
		=
		0.
	\end{equation*}
	Thus, $\tilde u \in \Uad$.
	Hence, we can use $v = \tilde u$ in \eqref{eq:OC_w_VI}.
	Similarly, we test $e_{\gamma_k}(u_k, w_k) = 0$ by $u_k - u$.
	Adding both relations yield
	\begin{align*}
		0
		&\ge
		\dual{S u - y_d}{S (u - \tilde u)}_Y
		+
		\dual{S u_k - y_d}{S (u_k - u)}_Y
		\\&\qquad+
		\alpha \dual{u - w}{u - \tilde u}_{L^2(\Omega)}
		+
		\alpha \dual{u_k - w_k}{u_k - u}_{L^2(\Omega)}
		+
		\gamma_k \dual{\pi(u_k - u_a)}{u_k - u}_{L^2(\Omega)}
		\\
		&\ge
		\norm{S(u_k - u)}_Y^2 + \dual{S u - y_d}{S (u_k - \tilde u)}_{Y}
		\\&\qquad
		+
		\alpha \norm{u_k - u}_{L^2(\Omega)}^2
		+
		\alpha \dual{u - w}{u_k - \tilde u}_{L^2(\Omega)}
		+
		\alpha \dual{w - w_k}{u_k - u}_{L^2(\Omega)}
		.
	\end{align*}
	Here, we have used
	$\dual{\pi(u_k - u_a)}{u_k - u}_{L^2(\Omega)} \ge 0$,
	which follows from
	$\pi(u_k - u_a) = 0$ a.e.\ on $\set{u_k \ge u_a}$,
	whereas
	$\pi(u_k - u_a), u_k - u \le 0$ a.e.\ on $\set{u_k < u_a}$.
	Together with the weak convergence $u_k \weakly \tilde u$
	this shows
	\begin{align*}
		0
		&\ge
		\limsup_{k \to \infty}
		\parens*{\norm{S(u_k - u)}_Y^2 + \alpha\norm{u_k - u}_{L^2(\Omega)}^2}
		\ge
		0,
	\end{align*}
	i.e., $u_k \to u$ in $L^2(\Omega)$.
	Since the limit $u = \tilde u$ is independent of the chosen subsequence,
	a usual subsequence-subsequence argument shows the convergence
	of the entire sequence.
\end{proof}
Finally,
we can approximate \eqref{eq:IOC} by a regularized problem
to prove enhanced regularity of local minimizers.
\begin{theorem}
	\label{thm:enhanced_regularity}
	Let $(\bar u, \bar w) \in L^2(\Omega) \times H_0^1(\Omega)$ be 
	a local minimizer of \eqref{eq:IOC}.
	Then, $-\Delta \bar w \in L^2(\Omega)$.
\end{theorem}
\begin{proof}
	We denote by $r > 0$ the radius of optimality
	and consider the regularized problem
	\begin{equation*}
		\label{eq:IOC_k}
		\tag{IOC$_k$}
		\begin{minproblem}
			[u \in L^2(\Omega), w \in H_0^1(\Omega)]
			{F(u, w) + \frac12 \norm{w - \bar w}_{H_0^1(\Omega)}^2}
			& e_{\gamma_k}(u, w) = 0, \\
			& w \ge w_a, \\
			& (u,w) \in B_r( (\bar u, \bar w) ).
		\end{minproblem}
	\end{equation*}
	Here, $\seq{\gamma_k}_{k \in \N} \subset \R^+$ with $\gamma_k \to \infty$
	is arbitrary.
	By arguing as in \cref{lem:existence_IOC},
	we obtain the existence of global solutions $(u_k, w_k)$ for this problem.
	Due to boundedness, we can extract a weakly convergent subsequence (without relabeling)
	with $(u_k, w_k) \weakly (\tilde u, \tilde w)$ in $L^2(\Omega) \times H_0^1(\Omega)$.
	Due to the compact embedding $H_0^1(\Omega) \embeds L^2(\Omega)$
	and \cref{lem:convergence_reg},
	we find
	$\tilde u = T(\tilde w)$
	and $u_k \to \tilde u$ in $L^2(\Omega)$.
	For $k$ large enough, $(T_{\gamma_k}(\bar w), \bar w)$ is feasible for
	\eqref{eq:IOC_k}, see \cref{lem:convergence_reg}.
	Thus,
	\begin{equation*}
		F(u_k, w_k) + \frac12 \norm{w_k - \bar w}_{H_0^1(\Omega)}^2
		\le
		F(T_{\gamma_k}(\bar w), \bar w).
	\end{equation*}
	Using the sequential weak lower semicontinuity of $F$
	and again \cref{lem:convergence_reg}
	yields
	\begin{align*}
		F(\tilde u, \tilde w)
		+
		\frac12 \norm{\tilde w - \bar w}_{H_0^1(\Omega)}^2
		&
		\le
		\liminf_{k \to \infty}
		F(u_k, w_k)
		+
		\frac12
		\limsup_{k \to \infty}
		\norm{w_k - \bar w}_{H_0^1(\Omega)}^2
		\\&
		\le
		\limsup_{k \to \infty}\parens*{
			F(u_k, w_k) + \frac12 \norm{w_k - \bar w}_{H_0^1(\Omega)}^2
		}
		\le
		F(\bar u, \bar w)
		.
	\end{align*}
	Since $(\bar u, \bar w)$ is a local minimizer
	and $(\tilde u, \tilde w) = (T(\tilde w), \tilde u) \in B_r( (\bar u, \bar w))$,
	$(\tilde u, \tilde w) = (\bar u, \bar w)$
	as well as
	$w_k \to \bar w$ in $H_0^1(\Omega)$
	follows.
	Thus, for $k$ large enough,
	the last constraint in \eqref{eq:IOC_k}
	is not binding
	and, consequently, $w_k$ is a local solution of the reduced problem
	\begin{equation*}
		\begin{aligned}
			\min_{w \in \Wad} 
			&\quad F(T_{\gamma_k}(w), w) + \frac12 \norm{w - \bar w}_{H_0^1(\Omega)}^2.
		\end{aligned}
	\end{equation*}
	Due to the Fréchet differentiability of $T_{\gamma_k}$
	from \cref{lem:regularized_OC}, $w_k$ satisfies
	the necessary optimality condition
	\begin{equation*}
		F'(u_k, w_k) (T_{\gamma_k}'(w_k) h, h)
		+
		\dual{w_k - \bar w}{h}_{H_0^1(\Omega)}
		\ge
		0
		\qquad\forall h \in \TT_{\Wad}(w_k).
	\end{equation*}
	Now,
	we introduce the adjoint $p_k \in L^2(\Omega)$
	as the unique solution of
	\begin{equation*}
		(\alpha + S\adjoint S) p_k + \gamma_k \pi'(u_k - u_a) p_k = -\alpha \, f'(u_k).
	\end{equation*}
	Testing with $p_k$ shows that $p_k$
	is bounded in $L^2(\Omega)$ uniformly w.r.t.\ $k \in \N$,
	since
	$f'(u_k) \to f'(\bar u)$.
	Next,
	we test the equation for $p_k$
	with $v := T'_{\gamma_k}(w_k) h$
	and
	\eqref{eq:derivative_T_gamma}
	with $p_k$.
	This gives
	\begin{align*}
		- \alpha \dual{f'(u_k)}{ T'_{\gamma_k}(w_k) h}_{L^2(\Omega)}
		&=
		\dual{(\alpha + S\adjoint S) p_k + \gamma_k \pi'(u_k - u_a) p_k}{ v }_{L^2(\Omega)}
		\\&=
		\dual{ p_k }{(\alpha + S\adjoint S) v + \gamma_k \pi'(u_k - u_a) v}_{L^2(\Omega)}
		\\&=
		\alpha \dual{p_k}{h}_{L^2(\Omega)}
		.
	\end{align*}
	Thus,
	the above optimality condition becomes
	\begin{equation*}
		\dual{\zeta-p_k}{h}_{L^2(\Omega)}
		+
		\dual{2 w_k - \bar w}{h}_{H_0^1(\Omega)}
		\ge
		0
		\qquad\forall h \in \TT_{\Wad}(w_k)
		.
	\end{equation*}
	Up to a subsequence, we have $p_k \weakly p$ in $L^2(\Omega)$.
	Due to $w_k \to \bar w$ in $H_0^1(\Omega)$,
	we can pass to the limit and obtain
	\begin{equation*}
		\dual{\zeta-p}{h}_{L^2(\Omega)}
		+
		\dual{\bar w}{h}_{H_0^1(\Omega)}
		\ge
		0
		\qquad\forall h \in \TT_{\Wad}(\bar w)
		.
	\end{equation*}
	This is an
	obstacle problem for $\bar w$ in $H_0^1(\Omega)$.
	We are going to apply the regularity result
	\cite[Theorem~4.32]{Troianiello1987}.
	Since we are dealing with the space $H_0^1(\Omega)$,
	we can choose $\Gamma = \emptyset$ therein.
	Thus, the requirement that $\Gamma$ is of class $C^1$
	is trivially satisfied.
	Since $\zeta - p \in L^2(\Omega)$
	and $-\Delta w_a \in L^2(\Omega)$,
	the application of
	\cite[Theorem~4.32]{Troianiello1987}
	yields
	\begin{equation*}
		\dualh{p-\zeta}{h}
		\leq \dualh{-\Delta\bar w}{h}
		\leq \dualh{\max(p-\zeta,-\Delta w_a)}{h}
	\end{equation*}
	for all $h\in H_0^1(\Omega)$ with $h \ge 0$.
	Due to $p-\zeta,-\Delta w_a\in L^2(\Omega)$
	this implies $-\Delta\bar w\in L^2(\Omega)$.
\end{proof}
We remark that $-\Delta \bar w \in L^2(\Omega)$
has to be understood in the distributional sense,
i.e.,
\begin{equation}
	\label{eq:distributional_laplacian}
	\dual{-\Delta \bar w}{\varphi}_{L^2(\Omega)}
	=
	\dual{\bar w}{\varphi}_{H_0^1(\Omega)}
\end{equation}
holds for all $\varphi \in C_c^\infty(\Omega)$.
Since both sides of the equation are continuous w.r.t.\ the $H_0^1(\Omega)$-norm of $\varphi$,
it even holds for all $\varphi \in H_0^1(\Omega) = \cl_{H_0^1(\Omega)}(C_c^\infty(\Omega))$.

We further note that we do not 
derive an equation
for the limit $p$ in the proof of \cref{thm:enhanced_regularity}.
We expect that this would lead to a system of C-stationarity.

\subsection{Linearization and M-stationarity}
\label{subsec:linearize_and_M}
After we have proved the regularity $-\Delta\bar w \in L^2(\Omega)$
in \cref{thm:enhanced_regularity},
we investigate
optimality conditions for \eqref{eq:IOC}.
We start by rewriting the VI \eqref{eq:OC_w_VI},
which characterizes the solution mapping $T$ of \eqref{eq:OC_w}.
To this end, we define the coercive operator
$A \in \LL(L^2(\Omega), L^2(\Omega))$,
\begin{equation*}
	A := \alpha \id + S\adjoint S,
\end{equation*}
and, consequently, the associated inner product
$\dual{\cdot\,}{\cdot}_A := \dual{A \, \cdot \,}{\cdot}_{L^2(\Omega)}$
is equivalent to the original inner product on $L^2(\Omega)$.
Using
\begin{equation*}
	S\adjoint(S u - y_d) + \alpha (u - w)
	=
	A \parens*{
		u - A^{-1}( S\adjoint y_d + \alpha w )
	}
	,
\end{equation*}
the VI \eqref{eq:OC_w_VI}
can be written as
\begin{equation*}
	\dual[\big]{
		u - A^{-1}( S\adjoint y_d + \alpha w )
	}{v - u}_{A}
	\ge
	0
	\qquad
	\forall v \in \Uad.
\end{equation*}

This is, in turn,
equivalent to
\begin{equation}
	\label{eq:lower-level_projection}
	u = \Proj_{\Uad}^A\parens*{
		A^{-1}( S\adjoint y_d + \alpha w )
	},
\end{equation}
where $\Proj_{\Uad}^A$ denotes the orthogonal projection
onto $\Uad$ w.r.t.\ the inner product $A$.
This representation of $T$
can be used to obtain the directional differentiability
of $T$ in a very convenient way.

\begin{lemma}
	\label{lem:differentiability_T}
	The solution mapping $T \colon L^2(\Omega) \to L^2(\Omega)$
	of \eqref{eq:OC_w}
	is directionally differentiable.
	For $\bar w,w \in L^2(\Omega)$,
	the directional derivative $z = T'(\bar w; w)$
	is characterized by the variational inequality
	\begin{equation}
		\label{eq:VI_direc_deriv_T}
		z \in \KK(\bar w),
		\qquad
		\dual[\big]{S\adjoint S z + \alpha (z - w)}{z - v}_{L^2(\Omega)}
		\le
		0
		\qquad
		\forall v \in \KK(\bar w).
	\end{equation}
	Here,
	\begin{equation*}
		\KK(\bar w)
		:=
		\set{
			v \in \TT_{\Uad}(\bar u) \given 
			\dual[\big]{S\adjoint(S \bar u - y_d) + \alpha (\bar u - \bar w)}{v}_{L^2(\Omega)}
			=
			0
		}
	\end{equation*}
	is the critical cone with $\bar u = T(\bar w)$.
	The mapping $L^2(\Omega) \ni w \mapsto T'(\bar w; w) \in L^2(\Omega)$
	is Lipschitz with constant $1$.
\end{lemma}
\begin{proof}
	It is well known that the set $\Uad$ is a polyhedric subset of $L^2(\Omega)$,
	see, e.g., \cite[Example~4.21~(1)]{Wachsmuth2016:2}.
	It follows,
	that the solution mapping $T$ of \eqref{eq:lower-level_projection}
	is directionally differentiable
	and the directional derivative $z = T'(\bar w; w)$
	satisfies
	\begin{equation*}
		z \in \KK(\bar w),
		\qquad
		\dual[\big]{
			z - A^{-1}( \alpha w )
		}{z - v}_{A}
		\le
		0
		\qquad
		\forall v \in \KK(\bar w),
	\end{equation*}
	where the critical cone $\KK(\bar w)$ is given by
	\begin{equation*}
		\KK(\bar w)
		:=
		\set{
			v \in \TT_{\Uad}(\bar u) \given 
			\dual[\big]{\bar u - A^{-1}( S\adjoint y_d + \alpha \bar w )}{v}_{A}
			=
			0
		}
		,
	\end{equation*}
	see
	\cite[Théorème~2.1]{Mignot1976}, \cite[Theorem~2]{Haraux1977}, \cite[Theorem~5.2]{Wachsmuth2016:2}.
	Utilizing the definition of $A$, the VI \eqref{eq:VI_direc_deriv_T} follows.
	The Lipschitz continuity follows by a suitable testing of \eqref{eq:VI_direc_deriv_T}.
\end{proof}
Using the extra regularity of local minimizers from \cref{thm:enhanced_regularity}
together with the differentiability of $T$ from the previous lemma,
we are in position to prove an optimality system of B-stationarity.
\begin{lemma}
	\label{lem:B_stat}
	Let $(\bar u, \bar w) \in L^2(\Omega) \times H_0^1(\Omega)$ be a local minimizer of \eqref{eq:IOC}.
	Then,
	\begin{equation}
		\label{eq:B_stat_dense}
		\dual{f'(\bar u)}{T'(\bar w; w)}_{L^2(\Omega)}
		+
		\dual{-\Delta \bar w + \zeta}{w}_{L^2(\Omega)}
		\ge
		0
		\qquad\forall w \in \cl_{L^2(\Omega)} \TT_{\Wad}(\bar w).
	\end{equation}
	Here, $\TT_{\Wad}(\bar w)$ is the tangent cone of $\Wad \subset H_0^1(\Omega)$
	at $\bar w$ w.r.t.\ the $H_0^1(\Omega)$-topology.
\end{lemma}
\begin{proof}
	Using the continuous solution operator $T$ of \eqref{eq:OC_w},
	$\bar w$ is a local solution to the reduced problem
	\begin{equation*}
		\begin{aligned}
			\min_{w \in \Wad} &\quad F(T(w), w) = f(T(w)) + \frac12 \norm{w}_{H_0^1(\Omega)}^2 + \dual{\zeta}{w}_{L^2(\Omega)}
			.
		\end{aligned}
	\end{equation*}
	Using the differentiability and Lipschitz continuity of $T$,
	this implies
	\begin{equation*}
		\dual{f'(\bar u)}{T'(\bar w; w)}_{L^2(\Omega)}
		+
		\dual{\bar w}{w}_{H_0^1(\Omega)}
		+
		\dual{\zeta}{w}_{L^2(\Omega)}
		\ge
		0
		\qquad\forall w \in \TT_{\Wad}(\bar w).
	\end{equation*}
	Finally, we use \eqref{eq:distributional_laplacian}
	and the continuity of $L^2(\Omega) \ni w \mapsto T'(\bar w; w) \in L^2(\Omega)$
	(see \cref{lem:differentiability_T})
	to arrive at \eqref{eq:B_stat_dense}.
\end{proof}
The next lemma characterizes the $L^2(\Omega)$-closure
of the $H_0^1(\Omega)$-tangent cone $\TT_{\Wad}(\bar w)$.
\begin{lemma}
	\label{lem:l2_closure}
	For all $\bar w \in \Wad$ we have
	\begin{equation*}
		\cl_{L^2}\TT_{\Wad}(\bar w)
		=
		\set{ v \in L^2(\Omega) \given v \ge 0 \aeon \set{\bar w = w_a} }
		.
	\end{equation*}
\end{lemma}
\begin{proof}
	We denote the set on the right-hand side by $L$.
	Due to $\TT_{\Wad}(\bar u) = \cl_{H_0^1(\Omega)} \cone (\Wad - \bar w)$,
	the inclusion ``$\subset$''
	follows from $\Wad - \bar w \subset L$,
	since $L$ is a closed convex cone in $L^2(\Omega)$.

	Now, let $v \in L$ be given.
	Note that $L$ is the $L^2(\Omega)$-tangent cone of
	\begin{equation*}
		\expandafter\hat \Wad := \set{w \in L^2(\Omega) \given w \ge w_a\aeon\Omega}.
	\end{equation*}
	Thus, for each $\varepsilon > 0$, there is $t > 0$ and $w \in \expandafter\hat \Wad$,
	such that
	$\norm{v - t \, (w - \bar w)}_{L^2(\Omega)} \le \varepsilon$.
	Due to the density of the embedding $H_0^1(\Omega) \embeds L^2(\Omega)$,
	there exists $\psi \in H_0^1(\Omega)$
	with $\norm{\psi - w}_{L^2(\Omega)} \le \varepsilon / t$.
	We define $\tilde w := \max(\psi, w_a)$.
	Then, we have $\tilde w \in H_0^1(\Omega)$,
	since
	$\max(0, w_a) \le \max(\psi, w_a) \le \abs{\psi} + \max(0,w_a)$
	and $\psi, \max(0,w_a) \in H_0^1(\Omega)$.
	In particular, $\tilde w \in \Wad$.
	Moreover,
	$\abs{\tilde w - w} = \abs{\max(\psi - w, w_a - w)} \le \abs{\psi - w}$
	implies
	$\norm{\tilde w - w}_{L^2(\Omega)} \le \norm{\psi - w}_{L^2(\Omega)} \le \varepsilon / t$.
	Thus,
	$\norm{v - t \, (\tilde w - \bar w)}_{L^2(\Omega)} \le 2\varepsilon$.
	Since $\varepsilon > 0$ was arbitrary,
	this shows
	$v \in \cl_{L^2(\Omega)} \TT_{\Wad}(\bar w)$.
\end{proof}
We remark that one has
\begin{equation*}
	\TT_{\Wad}(\bar w)
	=
	\set{ v \in H_0^1(\Omega) \given v \ge 0 \text{ q.e.\ on } \set{\bar w = w_a} }
\end{equation*}
and this characterization requires the tools of capacity theory,
see
\cite[Lemma~3.2]{Mignot1976}.
Note that the above proof does not use
this form of $\TT_{\Wad}(\bar w)$.

The next lemma shows that local optimality implies
that $(0,0,0)$ is a local minimizer of a linearized problem
similar to \eqref{eq:mpcc2}.

\begin{lemma}
	\label{lem:linearized_OC}
	Let $(\bar u, \bar w) \in L^2(\Omega) \times H_0^1(\Omega)$ be a local minimizer of \eqref{eq:IOC}
	and let
	$\bar\xi := S\adjoint(S \bar u - y_d) + \alpha (\bar u - \bar w)$.
	Then, $(0,0,0) \in L^2(\Omega)^3$
	is a (global) minimizer of
	\begin{equation}
		\label{eq:linearized_OC}
		\begin{minproblem}
			[u, w, \xi \in L^2(\Omega)]
			{\dual{f'(\bar u)}{u}_{L^2(\Omega)} + \dual{-\Delta \bar w + \zeta}{w}_{L^2(\Omega)}}
			& (\alpha + S\adjoint S) u - \alpha w - \xi = 0, \\
			& w \ge 0 \aeon \set{\bar w = w_a}, \\
			& u = 0 \aeon \set{\bar\xi > 0 }, \\
			& \xi = 0 \aeon \set{\bar u > u_a }, \\
			& 0 \le u \perp \xi \ge 0 \aeon \set{ \bar u = u_a, \bar\xi = 0}.
		\end{minproblem}
	\end{equation}
\end{lemma}
\begin{proof}
	Using the lower-level optimality conditions, one can see that
	$0\leq \bar u - u_a\perp \bar\xi\geq0$ holds a.e.\ in $\Omega$.
	The claim follows from \cref{lem:B_stat,lem:l2_closure}
	by using that $u = T'(\bar w; w)$ can be characterized by
	the VI \eqref{eq:VI_direc_deriv_T}
	and
	\begin{align*}
		\KK(\bar w) &=
		\set[\big]{u \in L^2(\Omega) \given
			u \ge 0 \aeon \set{\bar u = u_a, \bar \xi = 0}
			,\;
			u = 0 \aeon \set{\bar \xi > 0}
		}
		,
		\\
		-\KK(\bar w)\polar &=
		\set[\big]{\xi \in L^2(\Omega) \given
			\xi \ge 0 \aeon \set{\bar u = u_a, \bar \xi = 0}
			,\;
			\xi = 0 \aeon \set{\bar u > u_a}
		}
		.
	\end{align*}
\end{proof}

We want to consider a system for M-stationarity and other
stationarity systems for \eqref{eq:IOC}.
As we have only defined these stationarity systems for MPCCs,
we will consider the KKT reformulation of \eqref{eq:IOC},
which is obtained by replacing the lower-level by its KKT conditions.
\begin{equation*}
	\label{eq:KKTR}
	\tag{KKTR}
	\begin{minproblem}[u,\xi\in L^2(\Omega),w\in H_0^1(\Omega)]{
			F(u,w)
		}
		& w\geq w_a\quad\aeon\Omega,
		\\
		& S\adjoint(Su-y_d)+\alpha(u-w)-\xi =0,\\
		& 0 \leq u-u_a \perp \xi \geq0
		\quad\aeon\Omega.
	\end{minproblem}
\end{equation*}
This optimization problem is an instance of \eqref{eq:abstract_mpcc}.
Note that replacing the lower-level by its KKT condition
leads to an equivalent problem because the multiplier
$\xi$ depends continuously on $u$ and $w$.

Finally, we combine the linearization results with the results of \cref{sec:linear_mpcc}
to obtain M-stationarity of local minimizers of \eqref{eq:IOC}
under some assumptions.
\begin{theorem}
	\label{thm:mstat}
	Let $(\bar u, \bar w) \in L^2(\Omega) \times H_0^1(\Omega)$ 
	be a local minimizer of \eqref{eq:IOC}
	and let $\bar\xi := S\adjoint(S \bar u - y_d) + \alpha (\bar u - \bar w)$.
	We require that one of the following assumptions is satisfied.
	\begin{enumerate}
		\item
			\label{thm:mstat:nonneg}
			The operator $S\adjoint S$
			preserves non-negativity, i.e.,
			$u \ge 0$ a.e.\ implies $S\adjoint S u \ge 0$ a.e.,
			and $\bar u=u_a$ holds a.e.\ in $\Omega$.
		\item
			\label{thm:mstat:average}
			It holds $Y=\R$ and
			$Sv = \beta\dual{1}{v}_{L^2(\Omega)}$ for all $v\in L^2(\Omega)$,
			where $\beta>0$ is a constant.
	\end{enumerate}
	Then $(\bar u,\bar w,\bar\xi)$
	is an M-stationary point of \eqref{eq:KKTR}.
\end{theorem}
\begin{proof}
	According to \cref{lem:linearized_OC}, $(0,0,0)$ is a minimizer of the 
	linearized optimization problem \eqref{eq:linearized_OC}.
	This optimization problem is an instance
	of \eqref{eq:mpcc2} (but one does need to rescale $w$ by $\alpha$
	such that $\alpha w$ in the first constraint of \eqref{eq:linearized_OC} becomes $w$).
	In case~\ref{thm:mstat:nonneg}, \itemref{asm:operator_A:nonneg} 
	is satisfied, whereas 
	in case~\ref{thm:mstat:average}, \itemref{asm:operator_A:average_with_scalars}
	is satisfied.
	Thus, according to \cref{thm:linear_mstat_a,thm:linear_mstat_b},
	$(0,0,0)$ is an M-stationary point of \eqref{eq:linearized_OC}.

	When one writes down the M-stationarity systems of \eqref{eq:KKTR} 
	at $(\bar u,\bar w,\bar\xi)$ and of \eqref{eq:linearized_OC} at $(0,0,0)$, 
	it is easy to check that they are equivalent.
	Thus, $(\bar u,\bar w,\bar\xi)$ is an M-stationary
	point of \eqref{eq:KKTR}.
\end{proof}

\subsection{Example without strong stationarity}
\label{subsec:not_strong}
Finally, we present a class of problems
for which the (global) minimizer fails to be strongly stationary.
In addition to the general setting from \cref{subsec:IOC_setting},
we specify
\begin{align*}
	y_d &= 0 \in Y,
	&
	u_a = w_a &\equiv 0,
	&
	\zeta &\equiv 1,
	&
	f(u) &= -\int_\Omega u \, \d\omega
	.
\end{align*}
Further, we require that $S\adjoint S$
preserves non-negativity, i.e.,
$u \ge 0$ implies $S\adjoint S u \ge 0$.

For
$w \in L^2(\Omega)$ with $w \ge 0$,
the solution $u = T(w)$ of \eqref{eq:OC_w}
satisfies the VI \eqref{eq:OC_w_VI},
i.e.,
\begin{equation*}
	u
	=
	\Proj_{\Uad}^{L^2(\Omega)}( w - \alpha^{-1} S\adjoint S u )
	=
	\max\parens*{ 0,  w - \alpha^{-1} S\adjoint S u }
	.
\end{equation*}
From $u \ge 0$ we infer $S\adjoint S u \ge 0$,
thus, $u \le w$.
Hence, we have can estimate the objective for every feasible pair $(u,w) = (T(u), w)$
of \eqref{eq:IOC}
via
\begin{equation*}
	F(u,w)
	=
	\int_\Omega (w - u) \, \d\omega
	+
	\frac12 \norm{w}_{H_0^1(\Omega)}^2
	\ge
	0.
\end{equation*}
Together with $0 = T(0)$ and $F(0,0) = 0$,
it follows that
$(\bar u, \bar w) = (0,0)$
is the unique global minimizer of \eqref{eq:IOC}
in this setting.
Thus, $(\bar u,\bar w,\bar\xi)=(0,0,0)$
is the unique global minimizer of the equivalent \eqref{eq:KKTR}.

We check that this global minimizer is not strongly stationary.
If one applies \cref{def:stat} to \eqref{eq:KKTR},
the resulting system of strong stationarity is
\begin{align*}
	-1 + (\alpha\id + S\adjoint S)\adjoint p + \mu &= 0,
	\\
	1 - p + \lambda &= 0,
	\\
	- p + \nu &= 0,
	\\
	\lambda &\leq 0
	\quad\aeon\Omega,
	\\
	\mu&\leq0 \quad\aeon\Omega,
	\\
	\nu&\leq0 \quad\aeon\Omega.
\end{align*}
The third equation gives $p=\nu\leq 0$ a.e.\ in $\Omega$,
which implies 
$(\alpha \id+S\adjoint S)\adjoint p\leq0$ a.e.\ in $\Omega$.
Then
\begin{equation*}
	0 = 
	-1 + (\alpha\id + S\adjoint S)\adjoint p + \mu
	\leq -1 + \mu \leq -1
	\quad\aeon\Omega
\end{equation*}
follows, which is false.
Therefore, $(\bar u,\bar w,\bar\xi)=(0,0,0)$ is not strongly stationary.

However, \itemref{thm:mstat:nonneg} can be applied,
which tells us that $(\bar u,\bar w,\bar\xi)=(0,0,0)$
is an M-stationary point.
One possibility for the multipliers to the system of M-stationarity
can be found by setting $\nu=0$,
which implies $p=\nu=0$, $\mu=1$, $\lambda=-1$.
It is natural to ask, whether multipliers to the system of M-stationarity
can also be found with $\mu=0$.
This would imply $p = (\alpha + S\adjoint S)^{-1}(1)$ and $p=1+\lambda\leq1$.
However, if $\Omega=(0,1)$,
$Sv=\alpha\dual{1}{v}_{L^2(\Omega)}$ for all $v\in L^2(\Omega)$,
and $\alpha$ is sufficiently small, we have
$p=1/(\alpha+\alpha^2)>1$.
Thus, multipliers to the system of M-stationarity with $\mu=0$
cannot always be found.
However, since \itemref{thm:mstat:nonneg} is based
on \cref{thm:schinabeck} and strong stationarity does not hold,
we know from \cref{rem:more_multipliers} 
that there must be more than one choice for the multipliers $\bar\mu,\bar\nu$
such that the system of M-stationarity can be satisfied.

\section{Conclusion and outlook}
\label{sec:conclusion}

Generalizing the approach from \cite{Harder2020} to
Lebesgue spaces proved to be no simple task.
Interestingly, the A$_\forall$-stationarity 
seems to be a bigger problem than going from there to M-stationarity.
The A$_\forall$-stationarity requires to find KKT multipliers for
\eqref{eq:lpbeta} for all measurable $\beta$,
but even though these are convex and linear optimization problems,
finding KKT multipliers is not always possible, as shown by \cref{ex:no_fcq}. 

Nonetheless, in \cref{sec:IOC}, we demonstrated that
our results can be applied for a class of inverse optimal control problems.
Thus, we were able to show that local minimizers of these problems
are M-stationarity points,
and this result is also applicable
in situations where the local minimizer is not 
strongly stationary.

To the best of our knowledge, this was the first instance
that M-stationarity could be shown for MPCC-type problems in Lebesgue spaces.
It would be interesting to know, if these methods could be improved
to show M-stationarity for other MPCCs in Lebesgue spaces.
Furthermore, some ideas may be used to improve stationarity
conditions for MPCCs in the Sobolev space $H_0^1(\Omega)$, e.g., for the optimal control
of the obstacle problem.

%%fakesection: Bibliography
\printbibliography
\end{document}